\documentclass[onefignum,onetabnum]{siamonline190516}
\usepackage[utf8]{inputenc}
\usepackage[margin=1in]{geometry}

\usepackage{lipsum}
\usepackage{amsfonts}
\usepackage{graphicx}
\usepackage{epstopdf}
\usepackage{booktabs}       
\usepackage{algorithm}
\usepackage{algpseudocode}
\usepackage{algorithmicx}
\ifpdf
  \DeclareGraphicsExtensions{.eps,.pdf,.png,.jpg}
\else
  \DeclareGraphicsExtensions{.eps}
\fi

\usepackage{amsfonts}
\usepackage{graphicx}
\usepackage{amsmath}
\usepackage{amssymb}
\usepackage[shortlabels]{enumitem}
\usepackage{subfig}
\usepackage{epstopdf}
\usepackage{todonotes}
\usepackage{hyperref}
\usepackage{comment}
\usepackage[normalem]{ulem}
\usepackage{pifont}

\newcommand{\cmark}{\ding{51}}%
\newcommand{\xmark}{\ding{55}}%
\DeclareMathOperator*{\argmin}{argmin}
\newcommand{\bD}{\mathbf{D}}
\newcommand{\bL}{\mathbf{L}}
\newcommand{\bB}{\mathbf{B}}
\newcommand{\R}{\mathbb{R}}
\newcommand{\E}{\mathbb{E}}

\newcommand{\Pp}{\mathcal{P}}

\newcommand{\M}{\mathcal{M}}

\newcommand*\Laplace{\mathop{}\!\mathbin\bigtriangleup}
\newcommand\sbullet[1][.5]{\mathbin{\vcenter{\hbox{\scalebox{#1}{$\bullet$}}}}}
\newcommand{\blue}[1]{\textcolor{blue}{#1}}
\newcommand{\red}[1]{\textcolor{red}{#1}}
\newcommand{\magenta}[1]{\textcolor{magenta}{#1}}

\newsiamremark{remark}{Remark}
\newsiamremark{hypothesis}{Hypothesis}
\crefname{hypothesis}{Hypothesis}{Hypotheses}
\newsiamthm{claim}{Claim}
\newsiamthm{proposition}{Proposition}

\headers{Efficient Natural Gradient Descent Methods}{L. Nurbekyan, W. Lei, and Y. Yang}

\title{Efficient Natural Gradient Descent Methods for Large-Scale PDE-Based Optimization Problems\thanks{Submitted to the editors. 
\funding{L.~Nurbekyan was partially supported by AFOSR MURI FA 9550 18-1-0502 grant. Y.~Yang was partially supported by NSF grant DMS-1913129.}}}

\author{
Levon Nurbekyan\thanks{Department of Mathematics, UCLA (\email{lnurbek@math.ucla.edu}).}
\and Wanzhou Lei\thanks{Harvard University (\email{wanzhoulei@g.harvard.edu}).}
\and Yunan Yang\thanks{Institute for Theoretical Studies, ETH Z\"urich (\email{yunan.yang@eth-its.ethz.ch}).}
}

\usepackage{amsopn}

\ifpdf
\hypersetup{
  pdftitle={Efficient Natural Gradient Descent Methods},
  pdfauthor={L. Nurbekyan, W. Lei, and Y. Yang}
}
\fi

\begin{document}

\maketitle


\begin{abstract}
We propose efficient numerical schemes for implementing the natural gradient descent (NGD) for a broad range of metric spaces with applications to PDE-based optimization problems. Our technique represents the natural gradient direction as a solution to a standard least-squares problem. Hence, instead of calculating, storing, or inverting the information matrix directly, we apply efficient methods from numerical linear algebra. We treat both scenarios where the Jacobian, i.e., the derivative of the state variable with respect to the parameter, is either explicitly known or implicitly given through constraints. We can thus reliably compute several natural NGDs for a large-scale parameter space. In particular, we are able to compute Wasserstein NGD in thousands of dimensions, which was believed to be out of reach. Finally, our numerical results shed light on the qualitative differences between the standard gradient descent and various NGD methods based on different metric spaces in nonconvex optimization problems.
\end{abstract}



\section{Introduction}\label{sec:intro}

In this paper, we are interested in solving optimization problems of the form
\begin{equation}\label{eq:main}
    \inf_{\theta} f(\rho(\theta)),
\end{equation}
where $f$ is the objective/loss function and $\rho(\theta)$ is the state variable parameterized by $\theta$. We mainly consider $\rho(\theta)$ as a PDE-based forward model, and $f$ is a suitable discrepancy measure between the output of the forward model and the data. Inverse problems, such as the full waveform inversion (FWI), are classical examples of~\eqref{eq:main}. More recent examples are machine learning-based PDE solvers where $\rho(\theta)$ is a neural network with weights $\theta$ that approximates the solution to the PDE~\cite{raissi2019physics}. They are typical large-scale optimization problems either due to fine grids parameterization of the unknown parameter or large networks employed to approximate the solutions.

First-order methods, especially in neural network training, are workhorses of high-dimensional optimization tasks. One such approach is the gradient descent (GD) method, whose continuous analog is the following gradient flow equation
\[
    \dot{\theta}=-\partial_\theta f(\rho(\theta)).
\]
Although reasonably effective and computationally efficient, GD might suffer from local minima trapping, slow convergence, and sensitivity to hyperparameters. Consequently, first-order methods and some of their (stochastic and deterministic) variants are not robust and require a significant hyperparameter tuning on a problem-by-problem basis~\cite{yao2021adahessian}. Such performance is often explained by the lack of curvature information in the parameter updates. Many optimization algorithms have been developed to improve the convergence speed, such as Newton-type methods~\cite{xu2020second}, quasi-Newton methods~\cite{nocedal2006numerical}, and various acceleration techniques~\cite{nesterov1983method} including momentum-based methods~\cite{qian1999momentum}.

Recently, there has been a revival of second-order methods in the machine-learning community~\cite{xu2020second}. Significant developments include the AdaHessian~\cite{yao2021adahessian} and NGD~\cite{amari1985differential,martens2020new}. Both techniques incorporate curvature information into the parameter update. AdaHessian preconditions the gradient with an adaptive diagonal approximation to the Hessian~\cite{yao2021adahessian}. The diagonal approximation is estimated by an adaption of Hutchinson's trace estimator~\cite{hutchinson1990}. Consequently, one obtains an optimization method for~\cref{eq:main} with a similar observed convergence rate to Newton's method with a computational cost comparable to first-order methods. AdaHessian shows state-of-the-art performance across a range of machine learning tasks and is observed to be more robust and less sensitive to hyperparameter choices compared to several stochastic first-order methods~\cite{yao2021adahessian}.

A different approach is the natural gradient descent 
(NGD) method~\cite{amari1985differential,amari1998natural,pascanu2014revisiting,li2018natural,li2019wasserstein,mallasto2019formalization,martens2020new,shen2020sinkhorn}, which preconditions the gradient with the \textit{information matrix} instead of the Hessian; see~\eqref{eq:NGD_flow}. 
NGD performs the steepest descent with respect to the $\rho$-space, the \textit{``natural''} manifold where $\rho(\theta)$ resides, instead of the parameter $\theta$-space~\cite{amari1985differential,amari1998natural}. 
A Riemannian structure is imposed on the parameterized subset $\{\rho(\theta)\}$ and then pulled back into the $\theta$-space. 
NGD is sometimes also regarded as a generalized Gauss--Newton method~\cite{schraudolph2002fast,pascanu2014revisiting,martens2020new}, which has a faster convergence rate than GD. In particular, NGD can be interpreted as an approximate Netwon's method when the manifold metric and the objective function $f$ are compatible~\cite{martens2020new}. Other properties of NGD include local invariance with respect to the re-parameterization, robustness with respect to hyperparameter choices, ability to progress with large step-sizes, and enforcing a state-dependent positive semi-definite preconditioning matrix. Inspired by the success of NGD in machine learning, we aim to extend and apply it to PDE-based optimization problems, which are mostly formulated in proper functional spaces with rich flexibility in choosing the metric.


Mathematically, continuous-time NGD is the preconditioned gradient flow
\begin{equation}\label{eq:NGD_flow}
    \dot{\theta}=-G(\theta)^{-1} \partial_\theta f(\rho(\theta)),
\end{equation}
where $G(\theta)$ is the pull-back of a (formal) Riemannian metric in the $\rho$-space. It is often referred to as an information matrix and will be discussed in detail in~\Cref{sec:math_nat}. There are two options to discretize~\eqref{eq:NGD_flow}: explicit and implicit. An explicit Euler discretization of~\eqref{eq:NGD_flow} is
\begin{equation}\label{eq:NGD_explicit}
    \theta^{l+1}=\theta^l-\tau^l G(\theta^l)^{-1} \partial_\theta f(\rho(\theta^l)),\quad l=0,1,\ldots,
\end{equation}
where $\tau^l>0$ is the step size or learning rate. An implicit Euler discretization of~\eqref{eq:NGD_flow} gives rise to
\begin{equation}\label{eq:NGD_implicit}
    \theta^{l+1}=\argmin_{\theta}  \bigg\{ f(\rho(\theta)) + \frac{\langle G(\theta^l)(\theta-\theta^l),(\theta-\theta^l) \rangle}{2\tau^l} \bigg\},
\end{equation}
where $\langle\cdot,\cdot\rangle$ is the Euclidean inner product. If we denote by $d_\rho$ the divergence or distance generating $G(\theta)$, the second term in~\eqref{eq:NGD_implicit} is the leading-order Taylor expansion of $\frac{1}{2\tau} d_{\rho} (\rho(\theta), \rho(\theta^{l}) )^2$ at $\theta^l$. Thus, the solution of~\eqref{eq:NGD_implicit} agrees with
\begin{equation} \label{eq:nat_proximal}
    \theta^{l+1} = \argmin_{\theta}   \bigg\{ f(\rho(\theta)) + \frac{d_{\rho} (\rho(\theta), \rho(\theta^{l}) )^2}{2\tau^l} \bigg\},
\end{equation}
up to the first order.
Note that~\eqref{eq:nat_proximal} captures the underlying idea of the NGD: takeing advantage of the geometric structure to find a direction with a maximum descent in the $\rho$-space. In contrast, finding a maximum descent in the  $\theta$-space as done by the ``standard'' implicit GD is
\begin{equation} \label{eq:std_proximal}
    \theta^{l+1} = \argmin_{\theta}   \bigg\{ f(\rho(\theta)) + \frac{d_\theta(\theta, \theta^{l})^2}{2\tau^l} \bigg\},
\end{equation}
where $d_\theta$ is the chosen metric for the $\theta$-space. In this work, we focus on different $d_\rho$ and consider $d_\theta$ as the Euclidean distance for simplicity.
Intuitively, one may interpret it as a shift from the parametric $\theta$-space to the more ``natural'' $\rho$-space. Thus, the infinitesimal decrease in the value of $f$ and the direction of motion for $\rho$ on $\M$ at $\rho=\rho(\theta)$ are invariant under re-parameterizations~\cite{martens2020new}.

NGD has been proven to be advantageous in various problems in machine learning and statistical inference, such as blind source separation~\cite{amari1998adaptive}, reinforcement learning~\cite{peters2008natural} and neural network training~\cite{schraudolph2002fast,martens2012training,pascanu2014revisiting,martens2015optimizing,li2019affine,martens2020new,shen2020sinkhorn,lin2021wasserstein}. Further applications include solution methods for high-dimensional Fokker--Planck equations~\cite{li2019parametric,liu2022neural}. Despite its success in statistical inferences and machine learning, the NGD method is far from being a mainstream computational technique, especially in PDE-based applications. A major obstacle is its computational complexity. In~\eqref{eq:NGD_explicit}, explicit discretization of NGD reduces to preconditioning the standard gradient by the inverse of an often dense information matrix. The numerical computation is often intractable.

Existing works in the literature focused on explicit formulae~\cite{yang1998complexity}, fast matrix-vector products~\cite{schraudolph2002fast,martens2012training,pascanu2014revisiting,martens2020new}, and factorization techniques~\cite{martens2015optimizing} for natural gradients generated by the Fisher--Rao metric in the $\rho$-space where $\rho$ is the output of feed-forward neural networks. These methods exploit the structural compatibility of standard loss functions and the Fisher metric by interpreting the Fisher NGD as a generalized Gauss--Newton or Hessian-free optimization~\cite[Sec.~9.2]{martens2020new}. The computational aspects of feed-forward neural networks are also utilized since computations through the forward and backward passes are recycled. Thus, to the best of our knowledge, the neural-network community focuses on the Hessian approximation aspect in the context of feed-forward neural network models rather than the geometric properties of the forward-model-space. For the Wasserstein NGD (WNGD), \cite{li2019affine,arbel2019kernelized} rely on implicit Euler discretization, but their methods still suffer from accuracy issues due to the high dimensionality of the parameter space~\cite[Sec.~2]{shen2020sinkhorn}. A regularized WNGD was considered in~\cite{shen2020sinkhorn}. Unfortunately, by design, the method blows up when the regularization parameter decreases to zero, so it cannot compute the original WNGD. In~\cite{ying2021natural}, compactly supported wavelets were used to diagonalize the information matrix, which is limited to the periodic setting with strictly positive $\rho(\theta)$ and also certain smoothness assumptions for $\rho(\theta)$.

There are three main contributions in our work. First,  we depart from the Hessian approximation framework and adopt a more general \textit{geometric formalism} of the NGD. Our approach applies to a general metric for the state space, which can be independent of the choice of the objective function. As examples, we treat Euclidean, Wasserstein, Sobolev, and Fisher--Rao natural gradients in a single framework for an arbitrary loss function. We focus on the standard least-squares formulation of the NGD direction.  Second, we streamline the general NGD computation and develop two approaches to whether the forward model $\theta\mapsto \rho(\theta)$ is explicit or implicit. When the Jacobian $\partial_\theta \rho$ is analytically available, we utilize the (column-pivoting) QR decomposition for which a low-rank approximation can be directly applied if necessary~\cite{heavner2019building}. When $\partial_\theta \rho$ is only implicitly available through the optimization constraints, we employ iterative solution procedures such as the conjugate gradient method~\cite{metivier2013full} and utilize the adjoint-state method~\cite{plessix2006review}. This second approach shares the same flavor with the method of fast matrix-vector product for the Fisher--Rao NGD for neural network training~\cite{schraudolph2002fast,martens2012training,pascanu2014revisiting,martens2020new}, but it allows one to apply the general NGD to large-scale optimization problems (see~\Cref{subsec:FWI} for example). In particular, our method can perform high-dimensional Wasserstein NGD, which was believed to be out of reach in the literature~\cite[Sec.~1]{shen2020sinkhorn}. Last but not least, we use a few representative examples to demonstrate that the choice of metric in NGD matters as it can not only quantitatively affect the convergence rate but also qualitatively determine which basin of attraction the iterates converge to. 

The rest of the paper is organized as follows. In~\Cref{sec:math_nat}, we first present the general mathematical formulations of the natural gradient based on a given metric space $(\M, g)$ and how it contrasts with the standard gradient. We then discuss a few common natural gradient examples and how they can all be reduced to a standard $L^2$-based minimization problem on the continuous level. In~\Cref{sec:num_nat}, we demonstrate our general computational approaches under a unified framework that applies to any NGD method. The strategies concentrate on two scenarios regarding whether the Jacobian $\partial_\theta \rho$  is explicitly given or not, followed by~\Cref{sec:numerics} where we apply the proposed numerical strategies for NGD methods to optimization problems under these two scenarios. Conclusions and further discussions follow in~\Cref{sec:conclusions}.




\section{Mathematical formulations of NGD}\label{sec:math_nat}

We begin by discussing the NGD method in an abstract setting before focusing on the common examples. 

Assume that $\rho$ is in a Riemannian manifold $(\M,g)$, and $\theta$ is in an open set $\Theta \subset \R^p$. Furthermore, assume that the correspondence $\theta \in \Theta \mapsto \rho(\theta) \in \M$ is smooth so that there exist tangent vectors
\begin{equation}\label{eq:tangent_vec_gen}
    \Big \{\partial^g_{\theta_1} \rho(\theta),\partial^g_{\theta_2} \rho(\theta),\cdots,\partial^g_{\theta_p} \rho(\theta) \Big\} \subset T_{\rho} \M.
\end{equation}
The superscript $g$ in $\partial^g$ highlights the dependence of tangent vectors on the choice of the Riemannian structure $(\M,g)$. Furthermore, assume that $f:\M \mapsto \R$ is a smooth function and denote by $\partial^g_\rho f \in T_\rho \M$ its metric gradient; that is, for all smooth curves $t \mapsto \rho(t)$, we have
\[  \frac{d f(\rho(t))}{dt}=\big\langle \partial^g_\rho f(\rho(t)), \partial^g_t \rho(t) \big \rangle_{g(\rho(t))}. \]

Tangent vectors $\{\partial^g_{\theta_i} \rho\}_{i=1}^p$ incorporate fundamental information on how $\rho(\theta)$ traverses $\M$ when $\theta$ traverses $\Theta$. Indeed, an infinitesimal motion of $\theta$ along the coordinate $\theta_i$-axis in $\Theta$ induces an infinitesimal motion of $\rho$ along $\partial^g_{\theta_i} \rho$ in $\M$. More generally, if 
\[
\frac{d\theta}{dt} = \dot{\theta}= \eta = (\eta_1,\ldots, \eta_p)^\top,
\]
then
\[
   \partial^g_t \rho(\theta) =  \eta_1 \partial^g_{\theta_1} \rho+\cdots +\eta_p \partial^g_{\theta_p} \rho. 
\]
Consequently, we have that
\[
    \frac{d f(\rho (\theta)) }{dt}=\big\langle \partial^g_\rho f, \partial^g_t \rho(\theta) \big\rangle_{g(\rho(\theta))}= \big \langle \partial^g_\rho f,  \sum_{i=1}^p \eta_i \partial^g_{\theta_i} \rho \big\rangle_{g(\rho(\theta))}.
\]
Intuitively, to achieve the largest descent in the loss $f(\rho(\theta))$, we want to choose $\eta = (\eta_1,\ldots, \eta_p)^\top$ such that  $\partial^g_\rho f$ is as \textit{negatively} correlated with $\sum_{i=1}^p \eta_i \partial^g_{\theta_i} \rho$ as possible in terms of the given metric $g$. Thus, the NGD direction corresponds to the best approximation of $-\partial_\rho^g f$ by $\{\partial^g_{\theta_i} \rho \}$ in $T_\rho \M$:
\begin{equation}\label{eq:nat_grad_gen}
    \eta^{nat}=\argmin_\eta \bigg \|\partial_\rho^g f+\sum_{i=1}^p \eta_i \partial^g_{\theta_i} \rho\bigg \|^2_{g(\rho(\theta))}.
\end{equation}
In other words, the NGD corresponds to the evolution of $\theta$ that attempts to follow the \textit{manifold GD of $f$ on $\M$} as closely as possible. Since $\left(T_{\rho}\M,g\right)$ is an inner-product space where $g$ may depend on $\rho$, and $\rho$ depends on $\theta$, \eqref{eq:nat_grad_gen} implies that under the natural gradient flow, the direction of motion for $\rho$ on $\M$ is given by the $g(\rho(\theta))$-orthogonal projection of $-\partial^g_\rho f$ onto $\operatorname{span} \{\partial^g_{\theta_1} \rho,\ldots, \partial^g_{\theta_p} \rho \}$:
\begin{equation}\label{eq:nat_grad_dir_M_gen}
\partial_t^g \rho = \sum_{i=1}^p \eta^{nat}_i \ \partial^g_{\theta_i} \rho=:P \partial^g_\rho f.
\end{equation}
Since $\operatorname{span} \{\partial^g_{\theta_1} \rho,\ldots, \partial^g_{\theta_p} \rho \}$ is invariant under smooth changes of coordinates $\theta=\theta(\psi)$, we obtain that~\eqref{eq:nat_grad_dir_M_gen} is also invariant under such transformations. Additionally, the infinitesimal decay of the loss function is also invariant under smooth changes in the coordinates. Indeed,
\[
    \frac{d f(\rho(\theta))}{dt}=-\|P \partial^g_\rho f\|^2_{g(\rho(\theta))}.
\]
A critical benefit of these invariance properties is mitigating potential negative effects of a poor choice of parameterization by filtering them out (since the corresponding decrease in the loss function is parameter-invariant) and reaching $\argmin_{\rho \in \M} f(\rho)$ as quickly and as closely as possible. For the analysis of NGD based on this insight, we refer to \cite{martens2020new,liu2022neural} for more details.

\begin{remark}
When $\{\partial^g_{\theta_i} \rho \}$ are linearly dependent, the $\eta^{nat}$ in \eqref{eq:nat_grad_gen} is not unique, and we pick the one with the minimal length for computational purposes; that is, we replace $G^{-1}(\theta)$ by the Moore--Penrose pseudoinverse $G(\theta)^\dagger$ in~\eqref{eq:NGD_flow} and elsewhere. It is worth noting that this choice is crucial to guarantee convergence and generalization properties of the NGD method in some applications; see~\cite{zhang2019fast} for example. Alternatively, one may consider a damping variant of $G$; see~\Cref{subsec:damped}.
\end{remark}

To compare the natural gradient with the standard gradient $\partial_\theta f(\rho(\theta))$, first note that
\[ \frac{d f(\rho(\theta))}{dt}=\big \langle \partial^g_\rho f,  \sum_{i=1}^p \eta_i \partial^g_{\theta_i} \rho \big \rangle_{g(\rho(\theta))}=\sum_{i=1}^p \big \langle \partial^g_\rho f, \partial^g_{\theta_i} \rho \big \rangle_{g(\rho(\theta))}  \eta_i = \partial_{\theta} f(\rho(\theta)) \cdot  \eta.\]
Therefore, in a similar form with~\eqref{eq:nat_grad_gen}, the GD direction is the solution to 
\[
    \eta^{std}=\argmin_{\eta} \|\partial_
    \theta f(\rho(\theta))+\eta\|^2.
\]
In other words,  GD is the steepest descent in the $\theta$-space, whereas NGD is an approximation of the steepest descent in the $\rho$-space based on a given metric $g$. Furthermore, GD leads to
\begin{equation*}
\begin{split}
    \partial_t^g \rho =& \sum_{i=1}^p \eta^{std}_i \  \partial^g_{\theta_i} \rho = -\sum_{i=1}^p \big \langle \partial^g_\rho f, \partial^g_{\theta_i} \rho \big \rangle_{g(\rho(\theta))} \partial^g_{\theta_i} \rho,\\
    \frac{d f(\rho(\theta))}{dt}=&-\|\partial_\theta f(\rho(\theta))\|_2^2=-\sum_{i=1}^p \left| \big \langle \partial^g_\rho f, \partial^g_{\theta_i} \rho \big \rangle_{g(\rho(\theta))}\right|^2,
\end{split}
\end{equation*}
which are not necessarily invariant under coordinate transformations.

When $\{\partial^g_{\theta_i} \rho\}$ are linearly independent, we obtain that
\begin{equation}\label{eq:nat_std_relation}
    \eta^{nat}= - G(\theta)^{-1} \partial_\theta f(\rho(\theta))= G(\theta)^{-1} \eta^{std},
\end{equation}
where $G(\theta)$ is the \textit{information matrix} whose $(i,j)$-th entry is
\begin{equation}\label{eq:nat_kernel_matrix}
    G_{ij}(\theta)=  \big\langle \partial^g_{\theta_i} \rho, \partial^g_{\theta_j} \rho \big\rangle_{g(\rho(\theta))},\quad i,j=1,\ldots,p.
\end{equation}
Thus, an NGD direction is a GD direction preconditioned by the inverse of the information matrix. 

Since the information matrix $G(\theta)$ is often dense and can be ill-conditioned, direct application of~\eqref{eq:nat_std_relation} is prohibitively costly for high-dimensional parameter space; that is, large $p$. Our goal is to calculate $\eta^{nat}$ via the least-squares formulation~\eqref{eq:nat_grad_gen}, circumventing the computational costs from assembling and inverting the dense matrix $G$ directly.

\subsection{$L^2$ natural gradient}\label{subsec:L2nat_math}

In this subsection, we embed $\rho$ in the metric space $(\M,g)=\left( L^2(\R^d), \langle \cdot, \cdot \rangle_{L^2(\R^d)}\right)$. In this case, the tangent space $T_\rho \M = L^2(\R^d)$ for any $\rho \in \M$, and
\[
 \big \langle \zeta , \hat{\zeta} \big \rangle_{g(\rho)}= \int_{\R^d} \zeta(x) \hat{\zeta}(x) dx,\quad \forall \zeta, \hat{\zeta} \in T_\rho \M.
\]

The linear structure of $L^2(\R^d)$ is advantageous for developing differential calculus, and many finite-dimensional concepts generalize naturally. Indeed, the tangent vectors \eqref{eq:tangent_vec_gen} for a smooth mapping $\theta \in \Theta \mapsto \rho(\theta,\cdot) \in L^2(\R^d)$
are $\{\zeta_1,\zeta_2,\cdots,\zeta_p\}$ given by
\begin{equation}\label{eq:zeta's}
    \zeta_i(x)=\partial_{\theta_i} \rho(\theta,x),\quad i = 1,\ldots,p.
\end{equation}
The information matrix in \eqref{eq:nat_kernel_matrix} is given by
\[
    G^{L^2}_{ij}(\theta)=\int_{\R^d} \partial_{\theta_i}\rho(\theta,x)\partial_{\theta_j}\rho(\theta,x) dx,\quad i,j=1,2,\cdots,p.
\]
Next, for $f:L^2(\R^d) \mapsto \R$, we obtain that the $L^2$-derivative at $\rho$ is $\partial_\rho f(\rho) \in  L^2(\R^d)$ such that
\begin{equation}\label{eq:flat_grad}
    \lim \limits_{t\to 0}\frac{f(\rho+t \zeta)-f(\rho)}{t}=\int_{\R^d} \partial_\rho f(\rho)(x) \ \zeta(x) dx,\quad \forall \zeta \in L^2(\R^d).
\end{equation}
Thus, $\partial_\rho f$ is the commonly known derivative in the sense of calculus of variations. Finally, for smooth $\rho:\Theta \to L^2(\R^d)$ and $f:L^2(\R^d) \to \R$, formula \eqref{eq:nat_grad_gen} leads to the $L^2$ natural gradient
\begin{equation}\label{eq:nat_grad_L2}
    \eta^{nat}_{L^2}=\argmin_{\eta \in \R^p} \bigg \|\partial_\rho f+\sum_{i=1}^p \eta_i \zeta_i \bigg\|^2_{L^2(\R^d)}.
\end{equation}

The $L^2$ metric is not a typical choice for the NGD. Nevertheless, this metric is important as a basis for computing more complex NGDs. Additionally, see Section \ref{sec:L2_GN} for the connection between $L^2$-based NGD and the Gauss--Newton method.

\subsection{$H^s$ natural gradient}\label{subsec:Hsnat_math}

In this subsection, we assume that $\rho$ is embedded in the $L^2$-based Sobolev space $H^s(\R^d)$ for $s\in \mathbb{Z}$ (we return to the $L^2$ case if $s=0$). The metric space $(\M,g)= \left(H^s(\R^d), \langle \cdot, \cdot \rangle_{H^s(\R^d)} \right)$. Since this is also a Hilbert space, $T_\rho \M= H^s(\R^d)$ for all $\rho \in \M$, and
\begin{equation*}
    \langle \zeta, \hat{\zeta}\rangle_{g(\rho)}=  \langle \zeta, \hat{\zeta}\rangle_{H^s(\R^d)} =  \begin{cases}
    \int_{\R^d} \bD^s \zeta \cdot \bD^s \hat{\zeta}~dx, & s\geq 0,\\
    \int_{\R^d} \bD^{-s} \chi \cdot \bD^{-s} \hat{\chi}~dx, & s< 0,
    \end{cases}
    \quad \zeta,\hat{\zeta} \in T_\rho \M, 
\end{equation*}
where $\bD^s$ is the linear operator whose output is the vector of all the partial derivatives up to order $s$ for $s\geq0$. For $s<0$, we define $\chi = ((\bD^{-s})^*\bD^{-s})^{-1}  \zeta$ and $\hat{\chi} = ((\bD^{-s})^*\bD^{-s})^{-1}  \hat{\zeta}$. For example, $(\bD^{-s})^*\bD^{-s} = I-\Laplace$ if $s=-1$ and $I -\Laplace + \Laplace^2$ if $s=-2$~\cite{yang2022anderson}. Note that $\bD^{-s} ((\bD^{-s})^*\bD^{-s})^{-1}=((\bD^{-s})^*)^\dagger$ for $s<0$, where $^\dagger$ is the notation for pseudoinverse. Thus, we can rewrite 
\[ \langle  \zeta, \hat{\zeta} \rangle_{H^s(\R^d)}=\langle \bD^{-s}\chi,\bD^{-s}\hat{\chi} \rangle_{L^2(\R^d)}=\big \langle ((\bD^{-s})^*)^\dagger \zeta, ((\bD^{-s})^*)^\dagger\hat{\zeta} \big\rangle_{L^2(\Omega)},\quad \forall \zeta,\hat{\zeta}\in T_\rho \M. \]

For a smooth $\rho:\Theta \to H^s(\R^d)$, the tangent vectors are still $\{\zeta_i\}$ in \eqref{eq:zeta's} but now are considered as elements of $H^s(\R^d)$. This means that the information matrix $G^{H^s}(\theta)$ defined in~\eqref{eq:nat_kernel_matrix} is given by
\begin{equation*}
    G^{H^s}_{ij}(\theta)=\langle\partial_{\theta_i} \rho, \partial_{\theta_j} \rho \rangle_{H^s(\R^d)}=\begin{cases}
    \int_{\R^d} \bD^s \partial_{\theta_i} \rho(\theta,x) \cdot \bD^s \partial_{\theta_j} \rho(\theta,x)~dx, & s\geq 0,\\
    \int_{\R^d} ((\bD^{-s})^*)^\dagger \partial_{\theta_i} \rho(\theta,x) \cdot ((\bD^{-s})^*)^\dagger \partial_{\theta_j} \rho(\theta,x)~dx, & s< 0,
    \end{cases}
\end{equation*}
for $i,j=1,\ldots,p$. Note that $G^{H^s}$ is different from $G^{L^2}$ due to the inner product.

Next, we calculate the $H^s$ gradient of smooth $f:H^s(\R^d) \to \R$. For $s\geq 0$, we have that
\[
    \lim\limits_{t \to 0} \frac{f(\rho+t \zeta)-f(\rho)}{t}= \langle \partial^{H^s}_{\rho} f, \zeta \rangle_{H^s(\R^d)}  = \int_{\R^d} \bD^s \partial^{H^s}_{\rho} f \cdot \bD^s \zeta\  dx=\int_{\R^d} (\bD^s)^*\bD^s\ \partial^{H^s}_{\rho} f ~\zeta\ dx,
\]
and so from \eqref{eq:flat_grad} we obtain
\[
    \partial^{H^s}_{\rho} f =  \left((\bD^s)^* \bD^s \right)^{-1} \partial_\rho f,\quad s\geq 0.
\]
When $s<0$, under analogous assumptions with the case $s\geq 0$, we have that
\begin{equation*}
\begin{split}
    \lim\limits_{t \to 0} \frac{f(\rho+t \zeta)-f(\rho)}{t}=& \langle \partial^{H^s}_{\rho} f, \zeta \rangle_{H^s(\R^d)} = \int_{\R^d} ((\bD^{-s})^*)^\dagger \partial^{H^s}_{\rho} f\cdot ((\bD^{-s})^*)^\dagger \zeta  dx\\
    =&\int_{\R^d} (\bD^{-s})^\dagger ((\bD^{-s})^*)^\dagger \partial^{H^s}_{\rho} f\cdot  \zeta  dx=\int_{\R^d} ((\bD^{-s})^* \bD^{-s})^\dagger \partial^{H^s}_{\rho} f  \ \zeta dx.
\end{split}
\end{equation*}
Thus, from \eqref{eq:flat_grad}, we have 
\[
    \partial^{H^s}_{\rho} f = (\bD^{-s})^*\bD^{-s}   \partial_\rho f ,\quad s<0.
\]

Finally, for smooth $\rho:\Theta \to H^s(\R^d)$ and $f:H^s(\R^d) \to \R$, \eqref{eq:nat_grad_gen} leads to the $H^s$ natural gradient
\begin{equation}\label{eq:nat_grad_Hs}
    \eta^{nat}_{H^s}=\argmin_{\eta \in \R^p} \bigg \|\partial_\rho^{H^s} f+\sum_{i=1}^p \eta_i \zeta_i \bigg\|^2_{H^s(\R^d)}.
\end{equation}
For numerical implementation, we reduce this previous formulation into a least-squares problem in $L^2(\R^d)$. More specifically, for $s\geq 0$, \eqref{eq:nat_grad_Hs} can be written as
\[
    \eta^{nat}_{H^s}=\argmin_{\eta \in \R^p} \bigg \|\bD^s ((\bD^s)^*\bD^s)^{-1}\partial_\rho f+\sum_{i=1}^p \eta_i~ \bD^s \zeta_i \bigg\|^2_{L^2(\R^d)}.
\]
Furthermore, for $s<0$ we have that \eqref{eq:nat_grad_Hs} can be written as
\[
    \eta^{nat}_{H^s}=\argmin_{\eta \in \R^p} \bigg \|\bD^{-s} \partial_\rho f+\sum_{i=1}^p \eta_i~ ((\bD^{-s})^*)^\dagger \zeta_i \bigg\|^2_{L^2(\R^d)}.
\]
Both cases share the same form~\eqref{eq:nat_grad_Hs_L2} with $\bL=\bD^s$ for $s\geq 0$ and $\bL=((\bD^{-s})^*)^\dagger$ for $s<0$:
\begin{equation} \label{eq:nat_grad_Hs_L2}
      \eta^{nat}_{H^s}  =\argmin_{\eta \in \R^p} \bigg \|(\bL^*)^\dagger \partial_\rho f+\sum_{i=1}^p \eta_i~ \bL \zeta_i \bigg\|^2_{L^2(\R^d)}.
\end{equation}

\subsection{$\dot{H}^s$ natural gradient}\label{subsec:Hs_semi_nat_math}
Next, we consider the NGD with respect to the Sobolev semi-norm $\dot{H}^s$. For simplicity, we assume that $\rho$ is supported in a smooth bounded domain $\Omega\subset \R^d$. For $s>0$, we define the space $\dot{H}^s(\Omega)=\left\{\zeta \in H^s(\Omega):\int_\Omega \zeta=0 \right\}$ with the inner product
\[
\langle \zeta, \hat{\zeta} \rangle_{\dot{H}^s(\Omega)}=\langle \widetilde{\bD}^s \zeta, \widetilde{\bD}^s \hat{\zeta} \rangle_{L^2(\Omega)}=\int_{\Omega} \widetilde{\bD}^s \zeta \cdot \widetilde{\bD}^s \hat{\zeta} dx,\quad \forall \zeta,\hat{\zeta} \in \dot{H}^s(\Omega),
\]
where $\widetilde{\bD}^s$ is the linear operator whose output is the vector of all partial derivatives of positive order up to $s$. To consider the $\dot{H}^s$ natural gradient flows, we embed $\rho$ in $(\M,g)$, where
\[
    \M= \left\{ \rho \in H^s(\Omega):\int_\Omega \rho=1 \right\},\quad T_\rho \M = \dot{H}^s(\Omega),\quad 
    \langle \zeta,\hat{\zeta}\rangle_{g(\rho)}= \langle \zeta,\hat{\zeta}\rangle_{\dot{H}^s(\Omega)}, \quad \forall \zeta,\hat{\zeta} \in T_\rho \M.
\]
For a smooth $\rho:\Theta \to \M$, we still have that the tangent vectors are $\{\zeta_i\}$ as defined in \eqref{eq:zeta's}. Since $\int_{\Omega} \rho(\theta,x)dx=1$ for all $\theta \in \Theta$, we have that
\[
    \int_{\Omega} \zeta_i(x) dx = \int_{\Omega} \partial_{\theta_i}\rho(\theta,x) dx=\partial_{\theta_i} \int_{\Omega} \rho(\theta,x) dx=0, \quad i = 1,\ldots,p, 
\]
and thus $\{\zeta_i\} \subset T_{\rho}\M$. The information matrix~\eqref{eq:nat_kernel_matrix} for this case is $G^{\dot{H}^s}(\theta)$ given by
\[
    G^{\dot{H}^s}_{ij}(\theta)=\langle\partial_{\theta_i} \rho, \partial_{\theta_j} \rho \rangle_{\dot{H}^s(\Omega)}= \int_{\Omega}\widetilde{\bD}^s \partial_{\theta_i} \rho(\theta,x) \cdot \widetilde{\bD}^s \partial_{\theta_j} \rho(\theta,x) dx,\quad i,j=1,\ldots,p.
\]
On the other hand, for $f:\M \to \R$, we have that $\partial^{\dot{H}^s}_{\rho} f \in \dot{H}^s(\Omega)$ where $\forall \zeta \in T_\rho\M$,
\[
    \lim\limits_{t \to 0} \frac{f(\rho+t \zeta)-f(\rho)}{t}=\langle \partial^{\dot{H}^s}_{\rho} f,\zeta\rangle_{\dot{H}^s(\Omega)} = \int_{\Omega} \widetilde{\bD}^s \partial^{\dot{H}^s}_{\rho} f \cdot \widetilde{\bD}^s \zeta dx=\int_{\Omega}  (\widetilde{\bD}^s )^*\widetilde{\bD}^s \partial^{\dot{H}^s}_{\rho} f ~\zeta dx.
\]
The adjoint $(\widetilde{\bD}^s)^*$ is taken with respect to the $L^2(\Omega)$ inner product. Hence, based on~\eqref{eq:flat_grad},
\begin{equation}\label{eq:L2_seminorm_chain}
    \int_{\Omega}\left( \partial_{\rho} f- (\widetilde{\bD}^s )^*\widetilde{\bD}^s \partial^{\dot{H}^s}_{\rho} f \right) ~\zeta dx =0 ,\quad \forall \zeta \in T_\rho \M.
\end{equation}
Furthermore, denote by $\mathbf{1}$ the constant function that is equal to $1$ on $\Omega$. We then have that
\[
T_\rho \M=\operatorname{span}\{\mathbf{1}\}^\perp=\operatorname{ker}(\widetilde{\bD}^s)^\perp=\operatorname{Im}((\widetilde{\bD}^s)^*),
\]
where $^\perp$ is again taken with respect to the $L^2(\Omega)$ inner product. Hence, using the properties of adjoint operators, we obtain
\[    \partial^{\dot{H}^s}_{\rho} f=\left((\widetilde{\bD}^s)^*\widetilde{\bD}^s\right)^{\dagger} \partial_\rho f,\quad s>0.
\]

Next, we discuss the case $s<0$. As the dual space of $\dot{H}^{-s}(\Omega)$, the space $\dot{H}^s(\Omega)$ is equipped with the dual norm
\[
    \|\zeta\|_{\dot{H}^s(\Omega)}=\sup \left\{\langle \zeta, \phi \rangle :~\|\phi\|_{\dot{H}^{-s}(\Omega)}\leq 1 \right\}.
\]
Using the Poincar\'{e} inequality and the Riesz representation theorem, we obtain that for every $\zeta \in \operatorname{span}\{\mathbf{1}\}^\perp$, the map $\phi \mapsto \int_\Omega \zeta \phi$ is a continuous linear operator on $\dot{H}^{-s}(\Omega)$, and there exists a unique $\chi \in \dot{H}^{-s}(\Omega)$ such that
\[
    \int_\Omega \zeta~ \phi~ dx = \int_\Omega \widetilde{\bD}^{-s} \chi~ \widetilde{\bD}^{-s}\phi~ dx,\quad \forall \phi \in \dot{H}^{-s}(\Omega).
\]
Hence, $\zeta=(\widetilde{\bD}^{-s})^* \widetilde{\bD}^{-s} \chi$ together with the homogeneous Neumann boundary condition. Therefore, 
\[
    \|\zeta\|_{\dot{H}^{s}(\Omega)}=\|\widetilde{\bD}^{-s} \chi \|_{L^2}=\|\chi\|_{\dot{H}^{-s}(\Omega)}.
\]
Using similar arguments for the $s>0$ case, we obtain that
\[
    \langle \zeta,\hat{\zeta}\rangle_{\dot{H}^s(\Omega)}=\langle \widetilde{\bD}^{-s}\chi,\widetilde{\bD}^{-s} \hat{\chi}\rangle_{L^{2}(\Omega)}=\left\langle ((\widetilde{\bD}^{-s})^*)^\dagger \zeta, ((\widetilde{\bD}^{-s})^*)^\dagger  \hat{\zeta}\right\rangle_{L^{2}(\Omega)},\quad \forall \zeta,\hat{\zeta}\in \operatorname{span}\{\mathbf{1}\}^\perp.
\]
For more details on $\dot{H}^s(\Omega)$ where $s<0$, we refer to~\cite[Lecture~13]{ambrosio21}.

Next, we embed $\rho$ in space $\M= \left\{ \rho \in L^2(\Omega):\int_\Omega \rho=1 \right\}$ with $T_\rho \M = \operatorname{span}\{\mathbf{1}\}^\perp$ and
\[
    \langle \zeta, \hat{\zeta}\rangle_{g(\rho)}= \left\langle ((\widetilde{\bD}^{-s})^*)^\dagger \zeta, ((\widetilde{\bD}^{-s})^*)^\dagger  \hat{\zeta}\right\rangle_{L^{2}(\Omega)}, \quad \forall \zeta,\hat{\zeta}\in T_\rho \M.
\]
Furthermore, for a smooth function $f:\M \to \R$, we have that
\[
    \lim\limits_{t \to 0} \frac{f(\rho+t \zeta)-f(\rho)}{t}= \langle \partial^{\dot{H}^s}_{\rho} f,\zeta\rangle_{\dot{H}^s(\Omega)} 
    =\int_{\Omega}  ((\widetilde{\bD}^{-s})^* \widetilde{\bD}^{-s})^\dagger \partial^{\dot{H}^s}_{\rho} f ~\zeta dx.
\]
Together with~\eqref{eq:flat_grad}, we have
\[
    \int_{\Omega}\left( \partial_{\rho} f- ((\widetilde{\bD}^{-s})^* \widetilde{\bD}^{-s})^\dagger \partial^{\dot{H}^s}_{\rho} f \right) ~\zeta dx =0 ,\quad \forall \zeta \in T_\rho \M.
\]
After performing analysis similar to the $s>0$ case, we obtain that
\[    \partial_\rho^{\dot{H}^s} f= (\widetilde{\bD}^{-s})^* \widetilde{\bD}^{-s} \partial_\rho f,\quad s<0. \]

Finally, for both $s>0$ and $s<0$ cases, \eqref{eq:nat_grad_gen} leads to the $\dot{H}^s$ natural gradient
\begin{equation}\label{eq:nat_grad_Hsemi_s}
    \eta^{nat}_{\dot{H}^s}=\argmin_{\eta \in \R^p} \bigg \|\partial_\rho^{\dot{H}^s} f+\sum_{i=1}^p \eta_i \zeta_i \bigg\|^2_{\dot{H}^s(\R^d)},
\end{equation}
for smooth $\rho :\Theta \to \M$ and $f:\M \to \R$. As before, we can rewrite \eqref{eq:nat_grad_Hsemi_s} as a least-squares problem \begin{equation}\label{eq:nat_grad_Hsemi_s_L2}
        \eta_{\dot{H}^s}^{nat}=\argmin_{\eta \in \R^p} \bigg \| \left(\bL^* \right)^\dagger \partial_\rho f +\sum_{i=1}^p \eta_i~ \bL \zeta_i\bigg \|^2_{L^2(\Omega)},\quad \bL = \begin{cases} \widetilde{\bD}^s, & s>0 \\ ((\widetilde{\bD}^{-s})^*)^\dagger, & s<0 \end{cases}.
\end{equation}
Note that~\eqref{eq:nat_grad_Hsemi_s_L2} shares the same form with~\eqref{eq:nat_grad_Hs_L2}.

$H^s$ and $\Dot{H}^s$ natural gradients proved extremely useful for obtaining fast algorithms for solving the optimal transportation problem and related problems~\cite{jacobs2019solving,jacobs2020bf_ot,jacobs2021backforth}. The authors in these papers do not use the natural gradient descent formalism, but their methods are indeed Sobolev NGDs.

\subsection{Fisher--Rao--Hellinger natural gradient}\label{subsec:FSnat_math}
Here, we assume that $\rho$ is a strictly positive probability density function. We embed $\rho$ in $(\M,g)=(L^1(\R^d),g)$ where $T_{\rho}(\M)=L^2_{\rho^{-1}}(\R^d)$ and
\begin{equation*}
\langle \zeta,\hat{\zeta} \rangle_{g(\rho)} = \int_{\R^d}  \frac{\zeta(x)\hat{\zeta}(x)}{\rho(x)} dx,\quad \forall \zeta,\hat{\zeta} \in T_\rho \M.
\end{equation*}
This Riemannian metric is called the Fisher--Rao metric, and the distance induced by this metric is the Hellinger distance:
$d_{H}(\rho_1,\rho_2) \propto \|\sqrt{\rho_1}-\sqrt{\rho_2}\|_{L^2(\R^d)}$. Next, we will derive the natural gradient flow based on the Fisher--Rao metric, first introduced by Amari in~\cite{amari1998natural}.

For a smooth $\rho:\Theta \to \M$, we have that the tangent vectors are $\{\zeta_i\}$ in \eqref{eq:zeta's} but now considered as elements of $L^2_{\rho^{-1}}(\R^d)$. Therefore, the information matrix in~\eqref{eq:nat_kernel_matrix} becomes $G^{FR} (\theta) \in \R^{p\times p}$ where
\begin{equation*}
    G^{FR}_{ij}(\theta)=\int_{\R^d} \frac{\partial_{\theta_i} \rho(\theta,x) \partial_{\theta_j}\rho(\theta,x)}{\rho(\theta,x)} dx,\quad i,j=1,2,\ldots,p.
\end{equation*}
As before, $G^{FR}(\theta)$ is in general different from $G^{L^2}(\theta)$, $G^{H^s}(\theta)$ and $G^{\dot{H}^s}(\theta)$. 

Furthermore, for a smooth function $f:\M \to \R$, we have that
\begin{equation*}
    \lim\limits_{t \to 0} \frac{f(\rho+t \zeta)-f(\rho)}{t}=\int_{\R^d}  \frac{\partial^{FR}_{\rho} f \ \zeta}{\rho} dx,
\end{equation*}
and so from \eqref{eq:flat_grad} we obtain
\[
    \partial_\rho^{FR} f = \rho~\partial_\rho f.
\]
Finally, for smooth $\rho:\Theta \to \M$ and $f:\M \to \R$, \eqref{eq:nat_grad_gen} leads to the Fisher--Rao natural gradient
\begin{equation}\label{eq:nat_grad_FR}
    \eta^{nat}_{FR}=\argmin_{\eta \in \R^p} \bigg \|\partial_\rho^{FR} f+\sum_{i=1}^p \eta_i \zeta_i \bigg\|^2_{L^2_{\rho^{-1}}(\R^d)}.
\end{equation}
The $L^2$ least-squares formulation is
\begin{equation}\label{eq:nat_grad_FR_L2}
    \eta^{nat}_{FR}=\argmin_{\eta \in \R^p} \bigg \|\frac{\partial_\rho^{FR} f}{\sqrt{\rho}}+\sum_{i=1}^p \eta_i \frac{\zeta_i}{\sqrt{\rho}} \bigg\|^2_{L^2(\R^d)}
    = \argmin_{\eta \in \R^p} \bigg \|(\bL^*)^\dagger\partial_\rho f+\sum_{i=1}^p \eta_i~\bL \zeta_i \bigg\|^2_{L^2(\R^d)},
\end{equation}
where $\bL \zeta = \frac{1}{\sqrt{\rho}} \zeta$ and $(\bL^*)^\dagger \partial_\rho f = \sqrt{\rho}\ \partial_\rho f$.

\subsection{$W_2$ natural gradient}\label{subsec:W2nat_math}
We first revisit the WNGD method~\cite{li2018natural}. Denoting by $\Pp(\R^d)$ the set of Borel probability measures on $\R^d$, we first introduce the Wasserstein metric on the space $\Pp(\R^d)$.  Furthermore, for $\rho \in \Pp(\R^d)$ and a measurable function $f:\R^d \to \R^n$, we denote by $f_\sharp \rho \in \Pp(\R^n)$ the probability measure defined by
$$
    (f_\sharp \rho) (B)=\rho(f^{-1}(B)),\quad \forall B\subset \R^n~\text{Borel},
$$
and call it the pushforward of $\rho$ under $f$. Next, for any $\rho_1,\rho_2 \in \Pp(\R^d)$, we denote $\Gamma(\rho_1,\rho_2)$ as the set of all possible joint measure $\pi \in \Pp(\R^{2d})$ such that
\begin{equation*}
    \int_{\R^{2d}} \left( \phi(x)+\psi(y)\right)d\pi(x,y)=\int_{\R^d} \phi(x)d\rho_1(x)+\int_{\R^d} \psi(y)d\rho_2(y)  
\end{equation*}
for all $(\phi,\psi) \in L^1(\rho_1)\times L^1(\rho_2)$. The $2$-Wasserstein distance is defined as
\[
    W_2(\rho_1,\rho_2)=\left( \inf_{\pi \in \Gamma(\rho_1,\rho_2)} \int_{\R^{2d}} |x-y|^2 d\pi(x,y) \right)^{\frac{1}{2}}.
\]

Denoting by $\Pp_2(\R^d)$ the set of Borel probability measures with finite second moments, we have that $\left( \Pp_2(\R^d),W_2\right)$ is a complete separable metric space; see more details in \cite[Chapters 7]{villani03} and \cite[Chapters 7]{ags08}. More intriguingly, one can build a Riemannian structure on $\left( \Pp_2(\R^d),W_2\right)$. Our discussion is formal and we refer to \cite[Chapters 8]{villani03} and \cite[Chapters 8]{ags08} for rigorous treatments.

In short, tangent vectors in $\left( \Pp_2(\R^d),W_2\right)$ are the infinitesimal spatial displacements of minimal kinetic energy. More specifically, for a given $\rho \in \Pp_2(\R^d)$, we define the tangent space, $T_\rho \Pp_2(\R^d)$, as a set of all maps $v\in L^2_\rho(\R^d;\R^d)$ such that
\begin{equation}\label{eq:tanW2}
    \|v+w\|_{L^2_\rho(\R^d;\R^d)} \geq \|v\|_{L^2_\rho(\R^d;\R^d)},\quad \forall w \in L^2_\rho(\R^d;\R^d)\quad \text{s.t.}\quad \nabla \cdot (w \rho)=0,
\end{equation}
where $L^2_\rho(\R^d;\R^d)$ denotes the $\rho$-weighted $L^2$ space. When $\rho=1$, it reduces to the standard $L^2$. The divergence equation above is understood in the sense of distributions; that is,
\begin{equation*}
    \int_{\R^d} \nabla \phi(x) \cdot w(x) \ \rho(x) dx=0,\quad \forall \phi \in C^\infty_c(\R^d).
\end{equation*}

If we think of $\rho$ as a fluid density, then an infinitesimal displacement $\frac{dx}{dt}= \dot{x} = v(x)$
leads to an infinitesimal density change given by the continuity equation
\begin{equation}\label{eq:continuity}
    \frac{\partial \rho}{\partial t}=-\nabla \cdot (v \rho) . 
\end{equation}
Therefore, for a given $w$ such that $\nabla \cdot (w\rho)=0$, we have that both $\dot{x}  = v(x)$ and $\dot{x}  = v(x) + w(x)$
lead to the same continuity equation~\eqref{eq:continuity}.
Therefore, the evolution of the density is insensitive to the divergence-free vector fields, and we project them out leaving only a unique vector field with the minimal kinetic energy. The kinetic energy of a vector field $v$ is then defined as
\[
   \|v\|_{L^2_\rho(\R^d;\R^d)}^2= \int_{\R^d} |v(x)|^2 \rho(x) dx.
\]
For a given evolution $t \mapsto \rho(t,\cdot)$, such a ``distilled'' vector field $v$ is unique and incorporates critical geometric information on the spatial evolution of $\rho$.

Next, we define a Riemannian metric by
\begin{equation*}
    \langle v,\hat{v} \rangle_{g(\rho)} = \int_{\R^d} v(x)\cdot \hat{v}(x) \ \rho(x) dx,\quad v,\hat{v} \in T_\rho \Pp_2(\R^d).
\end{equation*}
Furthermore, a mapping $\theta \in \Theta \mapsto \rho(\theta,\cdot) \in \Pp(\R^d)$ is differentiable if for every $\theta \in \Theta$, there exists a set of bases  $\{v_i(\theta)\}\subset T_\rho \Pp_2(\R^d)$ such that
\begin{equation}\label{eq:v_i}
\lim\limits_{t\to 0} \frac{W_2 \left( \rho(\theta+t \eta),\left( I+t\sum_{i=1}^p \eta_i v_i(\theta)\right)\sharp \rho(\theta)\right)}{t}=0,\quad \forall \eta \in \R^p,
\end{equation}
where $I$ is the identity map. Thus,
\begin{equation}\label{eq:v's}
    \Big\{v_1,v_2,\cdots,v_p \Big\}= \Big\{\partial^W_{\theta_1} \rho, \partial^W_{\theta_2} \rho,\cdots, \partial^W_{\theta_p} \rho\Big\},
\end{equation}
are the tangent vectors in \eqref{eq:tangent_vec_gen} for the $W_2$ metric. Thus, the information matrix in~\eqref{eq:nat_kernel_matrix} becomes $G^{W} (\theta) \in \R^{p\times p}$ where
\begin{equation*}
    G^{W}_{ij}(\theta)=\int_{\R^d} v_i(x)\cdot v_j(x) \ \rho(x) dx,\quad i,j=1,2,\ldots,p.
\end{equation*}

For $f:\Pp_2(\R^d)\to \R$, the Wasserstein gradient at $\rho$ is then $\partial^W_\rho f(\rho) \in T_\rho \Pp_2(\R^d)$, such that
\begin{equation}\label{eq:Was_grad}
    \lim \limits_{t\to 0}\frac{f\left((I+tv)\sharp\rho\right)-f(\rho)}{t}=\int_{\R^d} \partial_\rho^W f(\rho)(x)\cdot v(x) \ \rho(x) dx,\quad \forall v \in T_\rho \Pp(\R^d).
\end{equation}
Thus, for a smooth $\rho: \Theta \to \Pp_2(\R^d)$ and $f:\Pp_2(\R^d) \to \R$, the $W_2$ NGD direction for $\theta$ is given by
\begin{equation}\label{eq:nat_grad_W2}
    \eta^{nat}_{W_2}=\argmin_{\eta \in \R^p} \bigg \|\partial^W_\rho f+\sum_{i=1}^p \eta_i v_i \bigg\|^2_{L^2_\rho(\R^d;\R^d)}.
\end{equation}

As seen in \eqref{eq:zeta's}, the $L^2$ derivatives and gradients are typically easier to calculate. Here, we discuss the relations between the $L^2$ and $W_2$ metrics that are useful for calculating the $W_2$ derivatives and gradients, i.e., $\{v_i\}$ and $\partial^W_\rho f$. We formulate the main conclusions in~\Cref{prop1}. 
\begin{proposition}\label{prop1}
Let $\{\zeta_i\}$ and $\{v_i\}$ follow~\eqref{eq:zeta's} and~\eqref{eq:v's}, respectively. The $\partial_\rho f$ and $\{\zeta_i\}$ in~\eqref{eq:nat_grad_L2} relate to the $\partial_\rho^W f$ and $\{v_i\}$ in~\eqref{eq:nat_grad_W2} as follows.
\begin{equation}\label{eq:Was-flat-connection}
    \partial_\rho^W f=\nabla \partial_\rho f,
\end{equation}
\begin{equation}\label{eq:zeta_v}
    v_i(\theta)=\argmin_v \Big\{\|v\|^2_{L^2_{\rho(\theta)}(\R^d;\R^d)}:~-\nabla \cdot (\rho(\theta) v )=\zeta_i(\theta) \Big\},\quad i=1,\ldots, p.
\end{equation}
\end{proposition}
\begin{proof}[Informal derivation]
Given a vector field $v$ and a small $t>0$, we have that $I+ t v $ is a first-order approximation of the trajectory below where $I$ is the identity function.
Note that in Lagrangian coordinates, $\dot{x} = v(x)$. Thus, from the continuity equation~\eqref{eq:continuity}, we have that
\begin{equation}\label{eq:infinitesimal_cty}
     \left(I+t v\right)\sharp \rho = \rho-t~\nabla \cdot (\rho v)  +o(t).
\end{equation}
Recall that $\zeta_i=\partial_{\theta_i} \rho$ and $v_i=\partial_{\theta_i}^W \rho$. Using this observation together with \eqref{eq:zeta's} and \eqref{eq:v_i}, we have
\begin{eqnarray*}
    \rho(\theta+t \eta)&=&\rho(\theta)+ t \sum_{i=1}^p \eta_i \zeta_i(\theta)+o(t),\\
    \rho(\theta+t \eta)&=&\rho(\theta)-t \sum_{i=1}^p \eta_i \nabla \cdot (\rho(\theta) v_i(\theta) )+o(t),
\end{eqnarray*}
for all $\eta \in \R^p$. By comparing the above two equations, we have
\begin{equation} \label{eq:v-zeta-relation}
    -\nabla \cdot (\rho(\theta) v_i(\theta))=\zeta_i(\theta),\quad 1\leq i \leq p.
\end{equation}
After taking~\eqref{eq:tanW2} into account, we obtain~\eqref{eq:zeta_v}.

Next, we establish a connection between $\partial_\rho f$ and $\partial^W_\rho f$. Combining \eqref{eq:flat_grad}, \eqref{eq:Was_grad}, \eqref{eq:infinitesimal_cty}-\eqref{eq:v-zeta-relation},
\begin{equation*}
        \int_{\R^d} \partial^W_\rho f(\rho)(x)\cdot v(x) \rho(x)dx=-\int_{\R^d} \partial_\rho f(\rho)(x) \nabla \cdot \left(\rho(x) v(x) \right) dx 
    = \int_{\R^d} \nabla \partial_\rho f(\rho)(x)  \cdot v(x)\ \rho(x) dx,
\end{equation*}
for all $v\in T_\rho \Pp_2(\R^d)$. Hence, we obtain~\eqref{eq:Was-flat-connection}.
\end{proof}

Similar to previous cases, we want to turn~\eqref{eq:nat_grad_W2} into an unweighted $L^2$ formulation. Using results in~\Cref{prop1}, we know that the Wasserstein tangent vectors at $\rho$ are velocity fields of minimal kinetic energy in $L^2_\rho(\R^d;\R^d)$. We first perform a change of variables
\begin{equation*}
    \Tilde v_i=\sqrt{\rho}~v_i,\quad i  = 1,\ldots, p,
\end{equation*}
where the set of $\{v_i\}$ follows~\eqref{eq:v's}. As a result, for each $i = 1,\ldots, p$, \eqref{eq:zeta_v} reduces to
\begin{equation} \label{eq:zeta_to_v}
\Tilde{v}_i(\theta)=\argmin\left\{\|\Tilde{v}\|^2_{L^2(\R^d;\R^d)}: \bB \Tilde{v} =\zeta_i(\theta)\right\}, \text{\quad where\quad } \bB  \Tilde{v}  = -\nabla \cdot \left(\sqrt{\rho(\theta )}~\Tilde{v} \right).
\end{equation}
We then have $\Tilde v_i = \bB ^\dagger \zeta_i$ for $i  = 1,\ldots, p$. Denote the adjoint operator of $\bB$ as $\bB^*$. Note that $\bB^* \eta =\sqrt{\rho} \nabla \eta$. Combining these observations with~\Cref{prop1}, formulation \eqref{eq:nat_grad_W2} becomes
\begin{equation}\label{eq:nat_grad_W2_L2}
\begin{split}
        \eta^{nat}_{W_2} &= \argmin_{\eta \in \R^p} \bigg \|\sqrt{\rho } \nabla \partial_\rho f+\sum_{i=1}^p \eta_i  \Tilde{v}_i \bigg\|^2_{L^2(\R^d;\R^d)}  =  \argmin_{\eta \in \R^p} \bigg \| \bB ^* \partial_\rho f+\sum_{i=1}^p \eta_i \bB ^\dagger \zeta_i \bigg\|^2_{L^2(\R^d;\R^d)} \\
        & = \argmin_{\eta \in \R^p} \bigg \| (\bL^*)^\dagger \partial_\rho f+\sum_{i=1}^p \eta_i \bL \zeta_i \bigg\|^2_{L^2(\R^d;\R^d)},\quad \text{where $\bL = \bB ^\dagger$.}
\end{split}
\end{equation}
We have reformulated the $W_2$ NGD as a standard $L^2$ minimization~\eqref{eq:nat_grad_W2_L2}.



\begin{remark}\label{rmk:Hm1_vs_W2}
Note that Wasserstein natural gradient is closely related to the $\dot{H}^{-1}$ natural gradient presented in \Cref{subsec:Hs_semi_nat_math}. Indeed, taking $s=-1$ in \eqref{eq:nat_grad_Hsemi_s_L2} we obtain that
\begin{equation*}
    \eta_{\dot{H}^{-1}}^{nat}= \argmin_{\eta \in \R^p}\|\nabla \partial_\rho f+\sum_{i=1}^p \eta_i (\nabla^*)^\dagger \zeta_i\|^2_{L^2(\Omega)},
\end{equation*}
which matches \eqref{eq:nat_grad_W2_L2} except that the weighted divergence operator $\bB$ defined in~\eqref{eq:zeta_to_v} is replaced with the unweighted divergence operator $-\nabla \cdot = \nabla^*$. When $\rho(\theta)\equiv 1$, these two operators coincide.

In principle, one may consider NGDs generated by the generalized operator
\begin{equation*}
    \bB_k \tilde{v}=-\nabla \cdot \left(\rho(\theta)^{k} \Tilde{v} \right),\quad \bL = (\bB_k)^\dagger,
\end{equation*}
where the case $k=0$ corresponds to the $\dot{H}^{-1}$ natural gradient and $k=1/2$ corresponds to the $W_2$ NGD. The term $\rho^k$ is often referred to as mobility in gradient flow equations~\cite{lisini2012cahn}.
\end{remark}

\begin{remark}
NGDs based upon the $L^2$~norm \eqref{eq:nat_grad_L2}, the $H^s$ norm~\eqref{eq:nat_grad_Hs}, the $\dot{H}^s$ norm~\eqref{eq:nat_grad_Hsemi_s}, the Fisher--Rao metric~\eqref{eq:nat_grad_FR} and the $W_2$ metric~\eqref{eq:nat_grad_W2} are similar in form but equipped with different underlying metric space $(\M,g)$ for $\rho$. All of them can be reduced to the same common form but with a different $\bL$ operator; see~\eqref{eq:nat_grad_L2}, \eqref{eq:nat_grad_Hs_L2}, \eqref{eq:nat_grad_Hsemi_s_L2}, \eqref{eq:nat_grad_FR_L2} and~\eqref{eq:nat_grad_W2_L2}, respectively. As a result, we expect that they may perform differently in the optimization process as NGD methods, which we will see later from numerical examples in~\Cref{sec:numerics}.
\end{remark}

\subsection{Gauss--Newton algorithm as an $L^2$ natural gradient}\label{sec:L2_GN}
Next, we give an example to show that the Gauss--Newton method, a popular optimization algorithm~\cite{nocedal2006numerical}, can be seen as an NGD method. More discussions on this connection can be found in~\cite{martens2020new}. Assume that $f$ measures the least-squares difference between the model $\rho(x;\theta)$ and the reference $\rho^*(x)$ distributions; that is,
\begin{equation} \label{eq:least-squares-obj}
    f(\rho(\theta))=\frac{1}{2}\int_{\Omega} |\rho(x;\theta)-\rho^*(x) |^2 dx,
\end{equation}
where $\Omega$ is the spatial domain. Thus, the problem of finding the parameter $\theta$ becomes
\begin{equation*}
    \inf_\theta f(\rho(\theta))=\inf_\theta\frac{1}{2}\int_{\Omega} | \rho(x ; \theta)-\rho^*(x)  | ^2 dx= \inf_\theta \frac{1}{2} \int_{\Omega}  | r(x;\theta) |^2 dx,\quad r(x;\theta) = \rho(x;\theta) - \rho^*(x).
\end{equation*}
We will denote $\rho(x;\theta)$ as $\rho(\theta)$ and $r(x;\theta)$ as $r(\theta)$.

The Gauss--Newton (GN) algorithm~\cite{nocedal2006numerical} is one popular computational method to solve this nonlinear least-squares problem. In the continuous limit, the algorithm reduces to the flow
\begin{equation}\label{eq:GN}
\dot{\theta}=\eta^{GN}=\argmin_{\eta \in \R^p} \left\|r(\theta)+\sum_{i=1}^p \partial_{\theta_i} r(\theta) \eta_i\right\|_{L^2(\Omega)}^2 = \argmin_{\eta \in \R^p} \left\|\rho(\theta)-\rho^*+\sum_{i=1}^p \partial_{\theta_i} \rho(\theta) \eta_i\right\|_{L^2(\Omega)}^2
\end{equation}
where we choose a mininal-norm $\eta$ if there are multiple solutions. The algorithm is based on a first-order approximation of the residual term 
$r(\theta+\eta)=r(\theta)+\sum_{i=1}^p \partial_{\theta_i} r(\theta) \eta_i+o(\eta)$.

A key observation is that \eqref{eq:GN} is precisely the $L^2$ natural gradient flow. Indeed, we have that
\begin{equation*}
    \lim \limits_{t \to 0} \frac{f(\rho+t \zeta)-f(\rho)}{t}=\int_{\Omega} \left(\rho(\theta)-\rho^* \right) \zeta(x) dx,
\end{equation*}
and therefore $\partial_\rho f(\rho)=\rho(\theta)-\rho^*$. As a result, \eqref{eq:nat_grad_L2} reduces to \eqref{eq:GN} precisely.

The convergence rate of the GN method is between linear and quadratic based on various conditions~\cite{nocedal2006numerical}. 
Typically, the method is viewed as an alternative to Newton's method if one aims for faster convergence than GD but does not want to compute/store the whole Hessian.

\begin{remark}
The $L^2$ natural gradient flow perspective of interpreting the GN algorithm suggests that mature numerical techniques for the GN algorithm are also applicable to \textit{general} NGD methods, including those we introduced earlier in~\Cref{sec:math_nat}. For further connections between GN algorithms, Hessian-free optimization and NGD see discussions and references in~\cite{schraudolph2002fast,pascanu2014revisiting,martens2015optimizing,martens2020new}.
\end{remark}
\begin{remark}\label{rmk:cost_discuss}
All natural gradient methods introduced in this section can be formulated as $\eta^{nat}  =\argmin_{\eta \in \R^p}  \|(\bL^*)^\dagger \partial_\rho f+\sum_{i=1}^p \eta_i~ \bL \zeta_i \|^2_{L^2}$,
while different metric space for $\rho$ gives rise to different operator $\bL$. The computational complexity of approximating $\bL$  and $(\bL^*)^\dagger$ determines the cost of implementing a particular NGD method. In general, $L^2$, $H^s$ and $\dot{H}^s$ NGDs are easier to implement as $\bL$ and $(\bL^*)^\dagger$ do not depend on $\rho$, and thus can be re-used from iteration to iteration once computed. On the other hand, for Fisher--Rao and Wasserstein NGDs, $\bL$ is $\rho$-dependent. If we have access to $\rho$ directly,  the Fisher--Rao information matrix only involves a diagonal scaling by $1/\rho$ compared to the $L^2$ information matrix. If we only have access to $\rho$ through an empirical distribution, there are also very efficient methods of estimating $G^{FR}$; see~\cite{martens2020new}. In contrast, the WNGD is the most expensive among all examples discussed in~\Cref{sec:math_nat}. Next, in~\Cref{sec:num_nat}, we will see that there are still efficient numerical methods to mitigate the computational challenges.
\end{remark}

\section{General computational approach} \label{sec:num_nat}
In this section, we discuss our general strategy to calculate the NGD directions. As mentioned earlier, our approach is based on efficient least-squares solvers since the problem of finding the NGD direction can be formulated as~\eqref{eq:nat_grad_gen}. In particular, we will introduce strategies when the tangent vector $\partial_\theta \rho$ cannot be obtained explicitly, which is the case for large-scale PDE-constrained optimization problems.  We will first describe the general strategies and then explain how to apply these techniques to different types of natural gradient discussed in~\Cref{sec:math_nat}. We will work in the discrete setting hereafter. 


By slightly abusing the notation, we assume that $\rho:\Theta \to \R^k$ is a proper discretization of $\theta \mapsto \rho(\theta)$ while $\Theta \subseteq \R^p$. Similarly, let $f:\R^k \to \R$ be a suitable discretization of $\rho \mapsto f(\rho)$. Hence, the standard finite-dimensional gradient and Jacobian, 
$ \partial_\rho f \in \R^k $ and $\partial_\theta \rho \in \R^{k \times p}$,
are discretizations of their continuous counterparts discussed in~\Cref{subsec:L2nat_math}. In particular, we denote the Jacobian
\begin{equation}\label{eq:Z's}
Z=(\zeta_1~\zeta_2~\cdots~\zeta_p)=\partial_\theta \rho, \quad \text{where }  \zeta_j=\partial_{\theta_j} \rho. 
\end{equation}
Without loss of generality, we always assume $k>p$. That is, we have more data than parameters.

\subsection{A unified framework}\label{subsec:unified}
For numerical computation, our main proposal is to translate the general formula~\eqref{eq:nat_grad_gen} and~\eqref{eq:nat_std_relation} for the NGD direction into a discrete least-squares formulation, given any Riemannian metric space $(\M, g)$.

Based on~\eqref{eq:nat_grad_L2}, the discrete $L^2$ natural gradient problem reduces to the least-squares problem
\[  \eta^{nat}=\argmin_{\eta \in \R^p}\|\partial_\rho f+Z \eta\|_2^2. \]
As we have seen in~\Cref{sec:math_nat}, besides $L^2$, the computation of the $H^s$, $\dot{H}^s$, Fisher--Rao, and WNGD directions can also be formulated as a least-squares problem
\begin{equation} \label{eq:unified}
    \eta^{nat}_L =\argmin_{\eta \in \R^p} \big \| (L^\top)^\dagger \partial_\rho f+L Z \eta \big \|_2^2 = \argmin_{\eta \in \R^p} \big \|(L^\top)^\dagger \partial_\rho f + Y \eta \big \|_2^2,\quad \text{where }  Y=LZ,
\end{equation}
for a matrix $L$ representing the discretization of the continuous operator $\bL$ for different metric spaces as discussed in~\Cref{sec:math_nat}. We regard~\eqref{eq:unified} as a unified framework since changing the metric space for the natural gradient only requires changing $L$ while the other components remain fixed.

Note that one can compute the standard gradient $\partial_\theta f = \partial_\theta \rho ^\top \partial_\rho f =  Z^\top \partial_\rho f$ by chain rule. From~\eqref{eq:unified}, we can also obtain the common formulation for the NGD as
\begin{equation}\label{eq:unified_kernel}
\begin{split}
    \eta^{nat}_L=&-(Z^\top L^\top L Z)^{-1} (Z^\top L^\top (L^\top)^\dagger\partial_\rho f)=-(Y^\top Y)^{-1} (Z^\top \partial_\rho f)\\
    =&-(Y^\top Y)^{-1} \partial_\theta f = -G_L^{-1} \partial_\theta f,
\end{split}
\end{equation}
where $G_{L} = Y^\top Y$ is the corresponding information matrix defined in~\eqref{eq:nat_kernel_matrix}. 

\begin{remark}
The unified framework~\eqref{eq:unified} is general and applies to cases beyond NGDs discussed in~\Cref{sec:math_nat}. For $\rho$ in a metric space $(\M, g)$ with a corresponding tangent space $T_\rho \M$, we have
\[
\langle \zeta_1, \zeta_2 \rangle_{g(\rho)} \approx \vec{\zeta_1}^\top A^g_\rho \  \vec{\zeta_2}, \quad \forall \zeta_1, \zeta_2 \in T_\rho \M,
\]
where $\vec{\zeta_1}$, $\vec{\zeta_2}$ denote the discretized $\zeta_1$, $\zeta_2$. A proper discretization that preserves the metric structure should yield a symmetric positive definite matrix $A^g_\rho$ that admits decomposition $A^g_\rho = L^\top L$. As a result, the discretization of~\eqref{eq:nat_std_relation} turns into the same formula as~\eqref{eq:unified}:
\begin{eqnarray*}
\eta_L^{nat} &=& - (Z^\top A^g_\rho Z )^{-1} (Z^\top \partial_\rho f ) \label{eq:nat_std_relation_disc} = - (Z^\top L^\top L Z )^{-1} (Z^\top L^\top  (L^\top )^\dagger \partial_\rho f )  \\ 
&=& \argmin_{\eta \in \R^p} \big\| (L^\top)^\dagger \partial_\rho f + Y \eta \big\|^2_2, \quad \text{where }  Y=LZ.
\end{eqnarray*}
The concrete form of $L$ will depend on the specific metric space $(\M, g)$. 

\end{remark}

Next, we will first assume that $L$ is given and discuss how to compute $\eta^{nat}_L$ provided whether the Jacobian $Z$ is available or not; see~\Cref{subsec:Z_avail} and~\Cref{subsec:Z_unavail}. Later in~\Cref{subsec:general_L}, we will comment on obtaining the matrix $L$ based on the natural gradient examples in~\Cref{sec:math_nat}.

\subsection{$Z$ available}\label{subsec:Z_avail}
When $Z$ is available, there are two main methods to compute $\eta_L^{nat}$. 

One may follow~\eqref{eq:unified_kernel} by first constructing the information matrix $G_L = Y^\top Y$ and then computing its inverse. This is a reasonable method when the number of parameters, i.e., $p$, is small, and $G_L$ is invertible. However, if $G_L$ is singular or has bad conditioning, it is \textit{more advantageous} to compute $\eta_L^{nat}$ following~\eqref{eq:unified}.
Note that the condition number of $G_L$ can be nearly the square of the condition number of $L$, making it more likely to suffer from numerical instabilities.

The second and also our recommended approach is to solve the least-squares problem~\eqref{eq:unified}. We may utilize the QR factorization to do so~\cite{golub1996matrix}. Assume that $Y = LZ$ has full column rank. Let $Y = QR$ where $Q$ has orthonormal columns and $R$ is an upper triangular square matrix. Thus,
\begin{equation}\label{eq:sol_Z_avail}
\eta_L^{nat}= - Y^\dagger (L^\top)^\dagger \partial_\rho f  = - R^{-1} Q^\top (L^\top)^\dagger \partial_\rho f.
\end{equation}
The additional computational cost of evaluating $\eta^{nat}_{L}$ after the QR decomposition is the backward substitution to evaluate $R^{-1}$ instead of inverting $R$ directly.

If the given model $\rho(\theta)$ allows us to write down how $\rho$ depends on $\theta$ analytically, then the Jacobian $\partial_\theta \rho$ is readily available. In such cases, we can directly solve~\eqref{eq:unified} using the QR decomposition to obtain the NGDs; see~\Cref{subsec:GaussianMixture} for a Gaussian mixture example.

We summarize the algorithm when the Jacobian $Z$ and the matrices $L,(L^\top)^\dagger$ are available; see~\Cref{subsec:general_L} for how to obtain $L$ and $(L^\top)^\dagger$ for examples presented in~\Cref{sec:math_nat} and~\Cref{sec:qr_low_rank} for discussions when $Y=LZ$ is rank-deficient.
\begin{algorithm}
\caption{Compute the NGD direction given $Z$\label{alg:Z_avail}, $L$, $(L^\top)^\dagger$ and $\partial_\rho f$.}
\begin{algorithmic}[1]
\State Compute $Y = LZ$.
\State Perform economy-size QR factorization: $[Q,R] = \texttt{qr}(Y)$.
\State Compute the NGD direction $\eta_L^{nat} = - R^{-1} Q^\top (L^\top)^\dagger \partial_\rho f$.
\end{algorithmic}
\end{algorithm}

\subsection{$Z$ unavailable}\label{subsec:Z_unavail}
Often, the model $\rho(\theta)$ is not available analytically, but the relationship between $\rho$ and $\theta$ is given implicitly via solutions of a system, e.g., a PDE constraint,
\begin{equation}\label{eq:h_constraint}
    h(\rho,\theta)= \bf 0,
\end{equation}
for some smooth $h:\R^k \times \R^p \to \R^k$ such that $\det (\partial_\rho h) \neq 0$. In such cases, the Jacobian $Z = \partial_\theta \rho$ in~\eqref{eq:Z's} is not readily available and has to be computed or implicitly evaluated. 

\subsubsection{The implicit function theorem and adjoint-state method}\label{subsec:adj_implicit}

Based on the first-order variation of~\eqref{eq:h_constraint}, the most direct option to proceed is to apply the implicit function theorem
\begin{equation}\label{eq:IFT}
 \partial_\rho h  \ \partial_\theta \rho     = \partial_\rho h \  Z = -\partial_\theta h.
\end{equation}
The above equation consists of $p$ linear systems in $k$ variables. If $\partial_\rho h$ has a simple format, or the size of $\theta$ is not too large, it could still be computationally feasible to first obtain $Z = \partial_\theta \rho$ by solving~\eqref{eq:IFT}, and then follow strategies in~\Cref{subsec:Z_avail} to compute the NGD.

However, if $p$ is large, a more efficient option is to use methods based on the so-called adjoint-state method~\cite{plessix2006review}. Note that $Z$ is the rate of change of the \textit{full state} $\rho$ with respect to $\theta$. Thus, if we only need the rate of change of $\rho$ \textit{along a specific vector} $\xi \in \R^k$, we do not need the whole $Z$; instead, we need $\xi^\top Z$ which can be calculated by solving only one linear system for each $\xi$.

Indeed, for a given $\xi \in \R^k$, let us consider the \textit{adjoint equation}
\begin{equation} \label{eq:adj}
    \lambda_\xi ^\top \partial_\rho h = \xi^\top \quad\Longleftrightarrow\quad \left( \partial_\rho h \right)^\top\lambda_\xi =  \xi.
\end{equation}
Combining \eqref{eq:IFT} and \eqref{eq:adj}, we obtain that
\begin{equation}\label{eq:adj_grad}
   Z^\top  \xi =   Z^\top \left( \partial_\rho h \right)^\top\lambda_\xi   =- \partial_\theta h^\top  \lambda_\xi.
\end{equation}
The vector $\lambda_\xi$ in \eqref{eq:adj} is called the \textit{adjoint variable} corresponding to the given vector $\xi$.

Here is an important example where we do not need the full $Z$. If we choose $\xi  = \partial_\rho f \in \R^{k}$, then~\eqref{eq:adj_grad} gives the standard gradient
\begin{equation}\label{eq:std_grad}
    \partial_\theta f(\rho(\theta)) = \partial_\theta \rho^\top   \partial_\rho f=    Z^\top \partial_\rho f  = - \partial_\theta h^\top \  \lambda_\xi,
\end{equation}
where $\lambda_\xi$ is the solution to~\eqref{eq:adj} with $\xi =\partial_\rho f \in \R^{k}$. This is a widely used method to efficiently evaluate the gradient of a large-scale optimization in solving PDE-constrained optimization problems originated from optimal control and computational inverse problems~\cite{plessix2006review}.

Next, we will explain in detail how to harness the power of the adjoint-state method to evaluate the general NGD directions through iterative methods.

\subsubsection{Krylov subspace methods}\label{subsec:Z_unavail_Krylov}
Given an arbitrary vector $\eta \in \R^p$, we may evaluate \begin{equation}\label{eq:nat_eval}
G_L \ \eta  = Z^\top L^\top L Z \   \eta 
\end{equation}
through the adjoint-state method even if we cannot access the information matrix $G_L$ since the Jacobian $Z$ is unavailable directly. Let $\widehat{\rho} \in \R^k$ be an arbitrary vector, and consider the following constrained optimization problem~\cite{metivier2013full}
\begin{equation}\label{eq:opt_L2nat}
    ~\min_\theta J(\rho(\theta)) = \rho^\top \widehat{\rho},\quad \mbox{s.t.}~~ h(\rho(\theta), \theta) = \bf 0.
\end{equation}
Note that this objective function $J(\rho(\theta))$ in~\eqref{eq:opt_L2nat} is different from the main objective function~\eqref{eq:main} but with the same constraint~\eqref{eq:h_constraint}. A direct calculation reveals that the gradient of $J(\rho(\theta))$ with respect to the parameter $\theta$ is $Z^\top \widehat{\rho}$. Therefore, if we set $\widehat{\rho} = L^\top L Z   \eta $, the gradient $$
\partial_\theta J(\rho(\theta)) = Z^\top \widehat{\rho} = Z^\top L^\top L Z \   \eta = G_L \ \eta,
$$
which is exactly what we aim to compute in~\eqref{eq:nat_eval}.

From the constraint $h(\rho(\theta),\theta) = \bf 0$ and its first-order variation~\eqref{eq:IFT}, we have 
$$
    \partial_\rho h \ Z\  \eta +   \partial_\theta h\  \eta = \bf 0.
$$
Thus, $Z \ \eta$ can be obtained as the solution to a linear system with respect to $ \gamma$:
\begin{equation}\label{eq:fwd_linear}
     \partial_\rho h\  \gamma = - \partial_\theta h \  \eta.
\end{equation}
Based on the adjoint-state method introduced in~Section~\ref{subsec:adj_implicit}, we can compute the gradient as 
$$
\partial_\theta J(\rho(\theta)) =  - \partial_\theta h^\top \  \lambda ,
$$ 
where $\lambda$ satisfies the adjoint equation below with a given $ \gamma$ that solves~\eqref{eq:fwd_linear},
\begin{equation} \label{eq:adj_linear}
\partial_\rho h^\top   \lambda =  \partial_\rho J  =  \widehat \rho =  L^\top L Z  \eta = L^\top L  \gamma.
\end{equation}

To sum up, with a fixed $\theta$ and the corresponding $\rho(\theta)$, we have an efficient way to evaluate the \textit{linear action} $\eta \mapsto G_L \eta$ for any given $\eta$ by three steps; see~\Cref{alg:linear_action}.
\begin{algorithm}
\caption{Evaluate the linear action $\eta \mapsto G_L \eta$ given  an arbitrary vector $\eta$.\label{alg:linear_action}}
\begin{algorithmic}[1]
\State Given the implicit constraint $h$, solve the linear system $\partial_\rho h\  \gamma = - \partial_\theta h \  \eta$ and obtain $\gamma$.
\State Given linear actions based on $L$ and $L^\top$, solve the linear system $\partial_\rho h^\top   \lambda = L^\top L \gamma$ and obtain $\lambda$.
\State Evaluate $- \partial_\theta h^\top \ \lambda$, which equals to $G_L\ \eta$.
\end{algorithmic}
\end{algorithm}


Given the linear action $\eta\mapsto G_L\eta$, we need to solve the linear system 
\begin{equation}\label{eq:G_L action}
G_L\ \eta_L^{nat} = -\partial_\theta f(\rho(\theta))
\end{equation}
to find the NGD direction $\eta_L^{nat}$. As seen in~\eqref{eq:std_grad}, we can obtain the right-hand side $-\partial_\theta f(\rho(\theta))$ through the adjoint-state method. One may then solve for $\eta_L^{nat}$ through iterative linear solvers based on the Krylov subspace methods~\cite{saad2003iterative}, e.g., the conjugate gradient method. We summarize all the steps above in~\Cref{alg:Z_unavail}.
\begin{algorithm}
\caption{Compute the NGD direction when $Z$ is not explicitly available.\label{alg:Z_unavail}}
\begin{algorithmic}[1]
\State Given the constraint $h$, solve the linear system $\left( \partial_\rho h \right)^\top\lambda = \partial_\rho f$ and obtain $\lambda$.
\State Compute the parameter gradient $\partial_\theta f(\rho(\theta)) = \partial_\theta \rho^\top   \partial_\rho f = - \partial_\theta h^\top \  \lambda$.
\State Obtain the linear action $\eta \mapsto G_L \eta$ following steps in~\Cref{alg:linear_action}.
\State Use the conjugate gradient method to solve for $\eta_L^{nat} $ where $G_L\,\eta_L^{nat} = -\partial_\theta f(\rho(\theta))$.
\end{algorithmic}
\end{algorithm}

One may use~\Cref{alg:Z_unavail} instead of~\Cref{alg:Z_avail} when $Z$ is available but the QR factorization of $Y = LZ$ is too costly, for instance, in some machine learning applications. Since ``wall-clock'' time can be highly affected by the implementation and the computer specification, in~\Cref{tab:propagation number}, we summarize the number of propagations per iteration among different methods~\cite{xu2020second}. For different NGDs, the cost of the linear action $\gamma \mapsto L^\top L \gamma$ varies, which we will discuss in~\Cref{subsec:general_L}.
\setlength{\tabcolsep}{6pt} 
\renewcommand{\arraystretch}{1.2} 
\begin{table}[!ht]
\centering
\caption{The number of propagations among different optimization methods.\label{tab:propagation number}}
\begin{tabular}{l|c|c|c}
\hline
 & GD  & NGD   & Newton's Method \\
 \hline 
Forward propagation $\theta\mapsto \rho(\theta)$                            & $1$ & $1$   & $1$             \\
 \hline 
Backward propagation $\xi \mapsto \partial_\theta \rho^\top\, \xi$            & $1$ & $1$   & $2$             \\
 \hline 
Linearized forward propagation $\omega \mapsto \partial_\theta \rho\, \omega$ & $0$ & $1^*$ & $1$   \\
\hline
 \multicolumn{4}{l}{\small $^*$For NGD, different choice of metric affects the complexity of the linearized forward solve.}
\end{tabular}
\end{table}

\subsection{Computation for natural gradient examples in~\Cref{sec:math_nat}}\label{subsec:general_L}
In~\Cref{subsec:Z_avail,subsec:Z_unavail}, we have shown how to compute the NGD direction $\eta_L^{nat}$ given $Z$ is easily available or not. Both strategies  require the matrix $L$, which depends on the particular metric space for the natural gradient. Next, we specify  the form of $L$ based on cases discussed in~\Cref{sec:math_nat}.

The $L^2$ case in~\Cref{subsec:L2nat_math} corresponds to $L=I$, the $k\times k$ identity matrix, while the Fisher--Rao--Hellinger natural gradient discussed in~\Cref{subsec:FSnat_math} corresponds to $L=\operatorname{diag}\left(1/\sqrt{\rho}\right) \in \R^{k\times k}$, which incurs $\mathcal{O}(k)$ more flops per iteration compared to the $L^2$ NGD method.
For the $H^s$ natural gradient discussed in~\Cref{subsec:Hsnat_math}, $L$ corresponds to proper discretization of $\bD^s$ (for $s> 0$) and $((\bD^{-s})^*)^\dagger$ (for $s<0$). Next, we give a few concrete examples. When $s=1$, $\bL = \bD^1 = [I, \nabla]^\top$ and $(\bL^*)^\dagger = \bD^1((\bD^1)^*\bD^1)^{-1} = [I, \nabla]^\top (I - \Laplace)^{-1}$. When $s=-1$, $\bL = ((\widetilde{\bD}^{-1})^*)^\dagger = [I, \nabla]^\top (I - \Laplace)^{-1}$ while $(\bL^*)^\dagger = [I, \nabla]^\top$. Similarly,  for the $\dot{H}^s$ natural gradient discussed in~\Cref{subsec:Hs_semi_nat_math}, $L$ should correspond to proper discretization of $\widetilde{\bD}^s$ (for $s> 0$) and $((\widetilde{\bD}^{-s})^*)^\dagger$ (for $s<0$). For instance, when $s=1$, $\bL = \widetilde{\bD}^1 =  \nabla$, and $(\bL^*)^\dagger = \widetilde{\bD}^1((\widetilde{\bD}^1)^*\widetilde{\bD}^1)^{-1} = \nabla (- \Laplace)^{-1}$; when $s=-1$, $\bL = ((\widetilde{\bD}^{-1})^*)^\dagger =  \nabla (- \Laplace)^{-1}$ while $(\bL^*)^\dagger = \nabla$. The symmetry between the cases of $H^s$/$\dot{H}^s$ and the cases of $H^{-s}$/$\dot{H}^{-s}$, $\forall s>0$, comes from the fact that they are dual Sobolev spaces. The computation of the natural gradient based on the $H^s$ and $\dot{H}^s$ metric can be efficiently computed. This is because there are fast algorithms for discretizing and computing the actions of the gradient and (inverse) Laplacian operators for periodic, Dirichlet and zero-Neumann boundary conditions in $\bL$ and $(\bL^*)^\dagger$~\cite{fortunato2020fast,yang2021implicit}.

Based on the unweighted reformulation~\eqref{eq:nat_grad_W2_L2}, computing the $W_2$ NGD discussed in~\Cref{subsec:W2nat_math} requires the discretization of $\bL = \bB^\dagger$. We can first discretize the differential operator $\bB$, denoted as $B$, and then compute $L = B^\dagger$, which can be used no matter the Jacobian $Z = \partial_\theta \rho$ is explicitly given or implicitly provided through the constraint~\eqref{eq:h_constraint}.
As an example, we describe how to obtain the matrix $L$ for the WNGD~\eqref{eq:nat_grad_W2_L2} in~\Cref{sec:WNGD} based on a finite-difference discretization of the differential operator. In~\Cref{rmk:Hm1_vs_W2}, we commented that when $\rho(x)$ is constant, WNGD reduces to $\dot{H}^{-1}$-based NGD. However, in general, the computation of the WNGD is more expensive than the $H^s$/$\dot{H}^s$ cases for two reasons. First, the information matrix $G$ and the operator $\bL$ for the WNGD are $\rho$-dependent, so in every iteration of the NGD method, one has to re-compute them, which incurs extra complexity. Second, as mentioned above, the computation of $H^s$/$\dot{H}^s$ NGD can be done through fast Fourier, or discrete cosine transforms (depending on the domain). It is, however, inapplicable to the Wasserstein case since it involves solving a \textit{weighted} differential equation. In~\Cref{sec:WNGD},  we use QR factorization to obtain $L = B^\dagger$ given $B$. We approximate $B$ using the finite-difference method, so $B^\top$ is very sparse. Using a multifrontal multithreaded sparse QR factoriazation~\cite{davis2011algorithm}, it has much better complexity than the conventional $\mathcal{O}(k^3)$. We summarize the observed computational costs of obtaining $L$ and $(L^\top)^{\dagger}$ for different NGD methods in~\Cref{tab:complexity_L}. See also~\Cref{fig:L and LL time} for the computational time comparison among different metrics.

After obtaining $L$ and $(L^\top)^{\dagger}$, the QR factorization of $Y = LZ$ followed by computing the natural gradient direction $\eta^{nat}_L$ based on~\eqref{eq:sol_Z_avail} will incur $\mathcal{O}(kp^2)$ flops if the Jacobian $Z$ is available; see~\Cref{fig:NGD time} for an observed computational time to obtain the NGD $\eta$ among different metrics for a case where $Z$ is analytically available (see Section~\ref{subsec:GaussianMixture}). When $Z$ is not analytic, such as from PDE (Section~\ref{subsec:FWI}) or neural network models (Section~\ref{sec:PINN}), we will see that the cost in computing NGDs among different methods is no longer dominated by the cost of computing $L$ and  $(L^\top)^{\dagger}$.

\setlength{\tabcolsep}{8pt} 
\renewcommand{\arraystretch}{1.2} 
\begin{table}[!ht]
\centering
\caption{Summary of the observed computational costs for linear actions $L$ and $(L^\top)^{\dagger}$ in~\eqref{eq:unified}.\label{tab:complexity_L}}
\begin{tabular}{cccccc}
\hline
 & $L^2$  & Fisher--Rao & $H^s$/$\dot{H}^s$, $s>0$ &   $H^s$/$\dot{H}^s$, $s<0$  & $W_2$ \\
 \hline
change over iteration &\xmark & \cmark &  \xmark & \xmark &\cmark \\
computing $v\mapsto Lv$ & $\mathcal{O}(k)$ & $\mathcal{O}(k)$ &   $\mathcal{O}(k)$ & $\mathcal{O}(k\log k)$ &  $\mathcal{O}(k^{1.25})$ \\
computing $v\mapsto (L^\top)^{\dagger}v$ & $\mathcal{O}(k)$ & $\mathcal{O}(k)$ & $\mathcal{O}(k\log k)$ & $\mathcal{O}(k)$ & $\mathcal{O}(k)$ \\
\hline
\end{tabular}
\end{table}

\begin{figure}
    \centering
    \subfloat[Evaluate $L$ and $(L^\top)^\dagger$ linear actions]{\includegraphics[width = 0.45\textwidth]{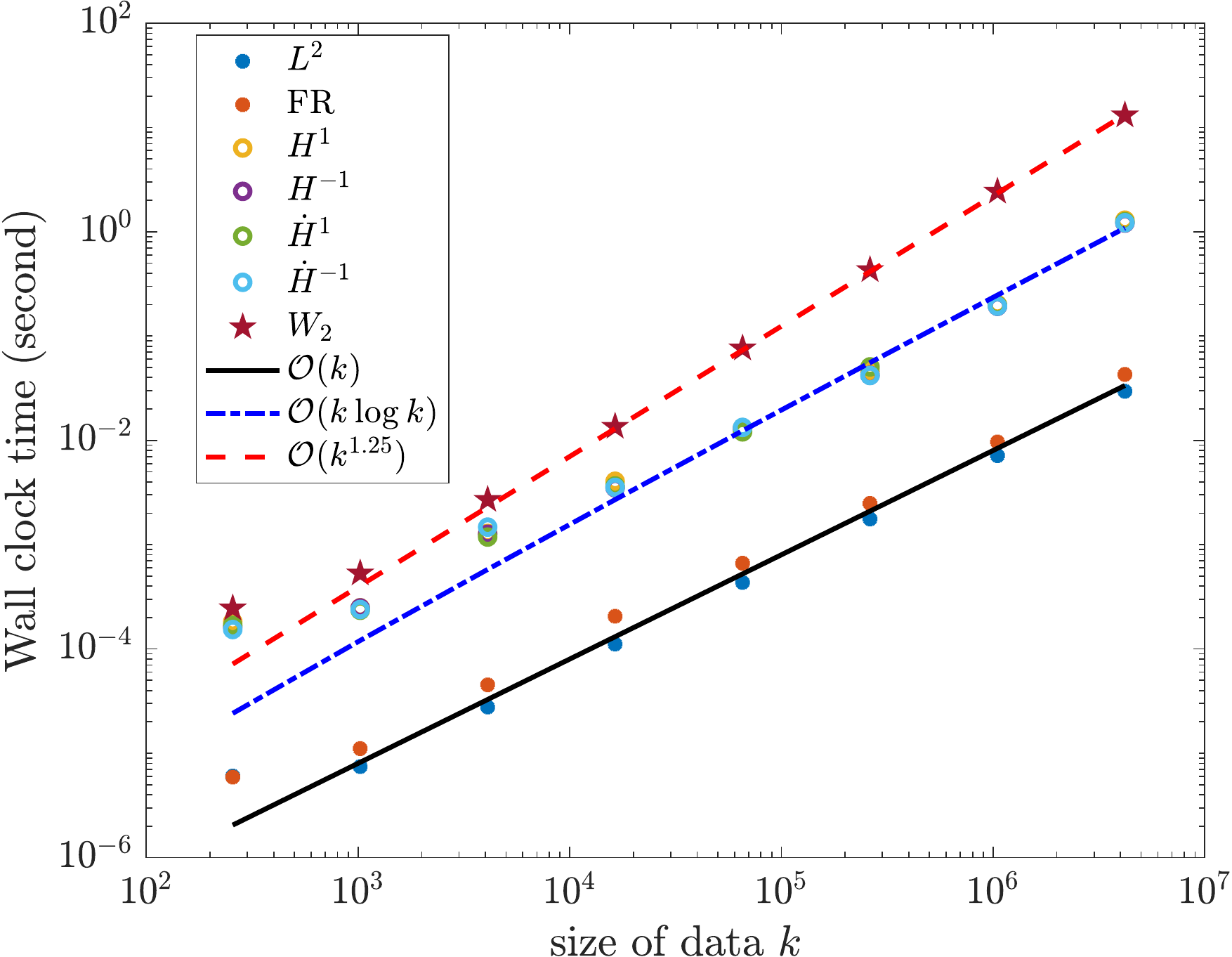} \label{fig:L and LL time}}
   \subfloat[Compute NGD direction $\eta$]{\includegraphics[width = 0.45\textwidth]{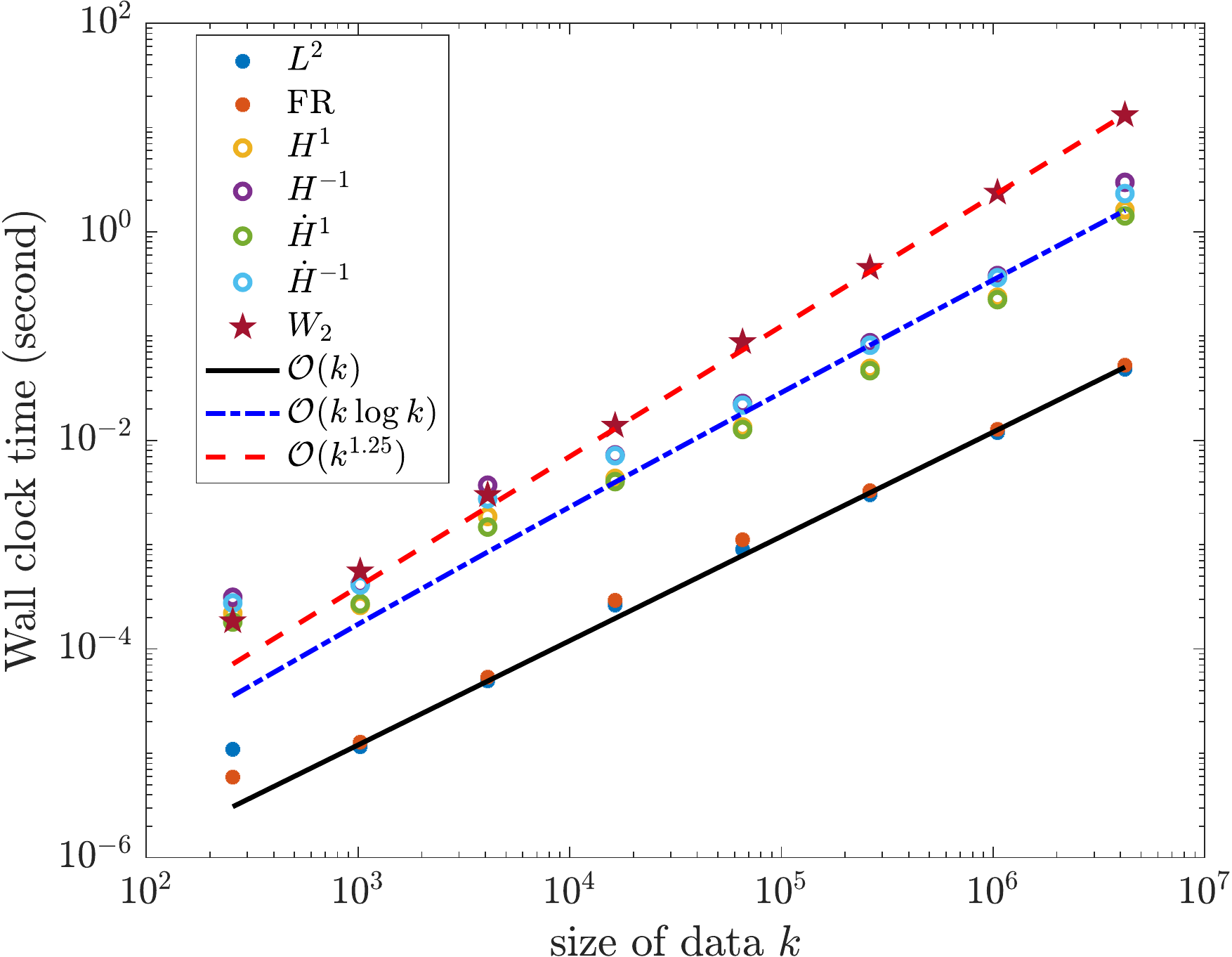} \label{fig:NGD time}}
   \caption{The observed wall clock time for evaluating $v\mapsto Lv$ and $v\mapsto (L^\top)^\dagger v$ linear actions (left) and for computing one NGD direction $\eta$ with a fixed $p$ (right) based on different metrics.\label{fig:NGD one run time}}
\end{figure}

\subsection{Extensions and variants}
In this section, we briefly comment on several practical variants of using the NGD method based on a particular choice of the data metric space.

\subsubsection{A damped information matrix}\label{subsec:damped}
If the discretized information matrix $G_L$ is rank deficient or ill-conditioned, one may consider rank-revealing QR factorization; see~\Cref{sec:qr_low_rank}. As an alternative approach, a damped information matrix in the form $G_\lambda = \lambda I+G_L$ is often used for numerical stability and to avoid extreme updates, where $\lambda$ is the damping parameter. One notable example is the Levenberg--Marquardt method as a damped Gauss--Newton method~\cite{schraudolph2002fast}, while the latter is 
equivalent to the $L^2$ NGD in our framework; see~\Cref{sec:L2_GN}.

Since the fundamental difference between GD and NGD lies in how one measures the distance between the potential next iterate and the current iterate, the damped version corresponds to choosing the next iterate based on a mixed metric from $\theta$-domain and $\rho$-domain. Indeed, in the implicit form~\cref{eq:nat_proximal,eq:std_proximal}, the damped version can be written as
\begin{equation}\label{eq:damped_proximal}
    \theta^{l+1} = \argmin_{\theta}   \bigg\{ f(\rho(\theta)) + \frac{ \lambda\, d_\theta( \theta,\theta^{l} )^2 +d_\rho( \rho(\theta) , \rho(\theta^{l})) ^2    }{2\tau} \bigg\}.
\end{equation}
When $d_\theta$ is the Euclidean metric on $\theta$-domain, we obtain the identity matrix $I$ in $G_\lambda$, but other choices of damping metric can also be considered. 

Alternatively, one can use another $\rho$-space metric to regularize instead of any metric on the $\theta$-space.
For example, let $d_{\rho_2}$ be the main natural gradient metric and $d_{\rho_1}$ be the regularizing natural gradient metric. The next iterate  obtained in the implicit Euler scheme is given by
\begin{equation}\label{eq:damped_proximal_2}
    \theta^{l+1} = \argmin_{\theta}   \bigg\{ f(\rho(\theta)) + \frac{ \lambda\, d_{\rho_1}( \rho(\theta) , \rho(\theta^{l})) ^2  +d_{\rho_2}( \rho(\theta) , \rho(\theta^{l})) ^2    }{2\tau} \bigg\},
\end{equation}
while the damping parameter $\lambda$ determines the strength of regularization. We comment that the $H^1$ natural gradient can be seen as the $\dot{H}^1$ natural gradient damped by the $L^2$ natural gradient.

\subsubsection{Mini-batch NGD}\label{subsubsec:minibatch}

Similar to mini-batch GD, one can also use mini-batch NGD by  computing the \textit{natural} gradient of the objective function with respect to a subset of the data $\rho$. Consider a random sketching matrix $S\in \R^{k'\times k}$,  $k' < k$. Each row of $S$ has at most one nonzero entry $1$. Thus, $S\rho \in \R^{k'}$ is the mini-batch data. The objective function also becomes $f(S\rho(\theta))$.

The mini-batch NGD can find the next iterate $ \theta^{l+1}$ implicitly through
\[
    \theta^{l+1} = \argmin_{\theta}   \bigg\{ f({ S} \rho(\theta)) + \frac{ d_\rho( {S}\rho(\theta) , {S} \rho(\theta^{l}) )^2    }{2\tau} \bigg\},
\]
where $d_\rho$ is the $\rho$-space metric. It is equivalent to changing the data metric from $d_\rho( \boldsymbol{\cdot} , \boldsymbol{\cdot} )$ to a random pseudo metric $d_\rho( { S} \boldsymbol{\cdot}  , { S} \boldsymbol{\cdot} )$. The information matrix and the NGD direction are
\[
G =   Z^\top {S} ^\top L^\top   L {S}  Z ,\qquad \eta = G^{-1} \partial_\theta f(S\rho(\theta)) ,
\]
where $L$ depends on $d_\rho( { S} \boldsymbol{\cdot}  , { S} \boldsymbol{\cdot} )$ and $Z$ is the Jacobian. Note that $S$ changes over iterations.

Also, we remark that $SZ \in \R^{k'\times p}$ can be seen as a random sketching of the Jacobian matrix $Z$. If $Z$ is low-rank, the column space of $SZ \in \R^{k'\times p}$ can be a close approximation to the column space of $Z$, but $SZ$ is much smaller in size. 
See~\Cref{sec:column_Z} where similar techniques from random linear algebra can help explore the column space of $Z$ and further reduce the computational cost.


\section{Numerical results} \label{sec:numerics}
In this section, we present three optimization examples to illustrate the effectiveness of our computational strategies for NGD methods. We first present the parameter reconstruction of a Gaussian mixture model where the Jacobian $\partial_\theta \rho$ is analytically given. 
Our second example is to solve the 2D Poisson equation using the physics-informed neural networks (PINN)~\cite{raissi2019physics}, where the Jacobian $\partial_\theta \rho$ can be numerically obtained through automatic differentiation. We then present a large-scale waveform inversion, a PDE-constrained optimization problem where the Jacobian $\partial_\theta \rho$ is not explicitly given. Using our computational strategy proposed in~\Cref{subsec:Z_unavail}, we can efficiently implement the NGD method based on a general metric space. The first example shows that various (N)GD methods converge to different stationary points of a nonconvex objective function. The last two tests illustrate that different (N)GD methods have various convergence rates. Both phenomena are interesting as they indicate that one may achieve global convergence or faster convergence by choosing a proper metric space $(\M,g)$ that fits the problem.

\begin{figure}
\centering
\subfloat[GD]{\includegraphics[width = 0.166\textwidth]{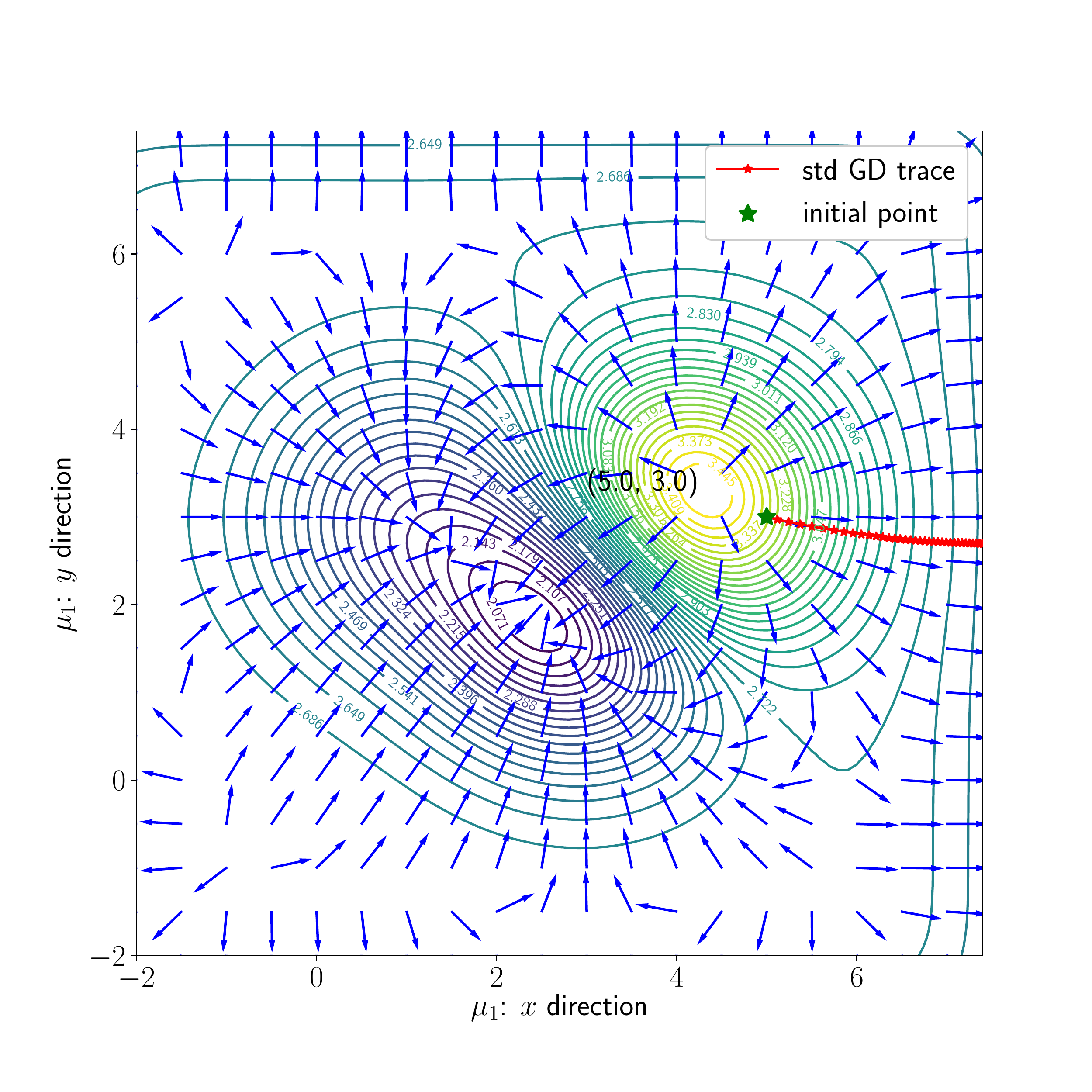}\label{fig:mu1_two_mixture_std_init1}}
\subfloat[$L^2$ NGD]{\includegraphics[width = 0.166\textwidth]{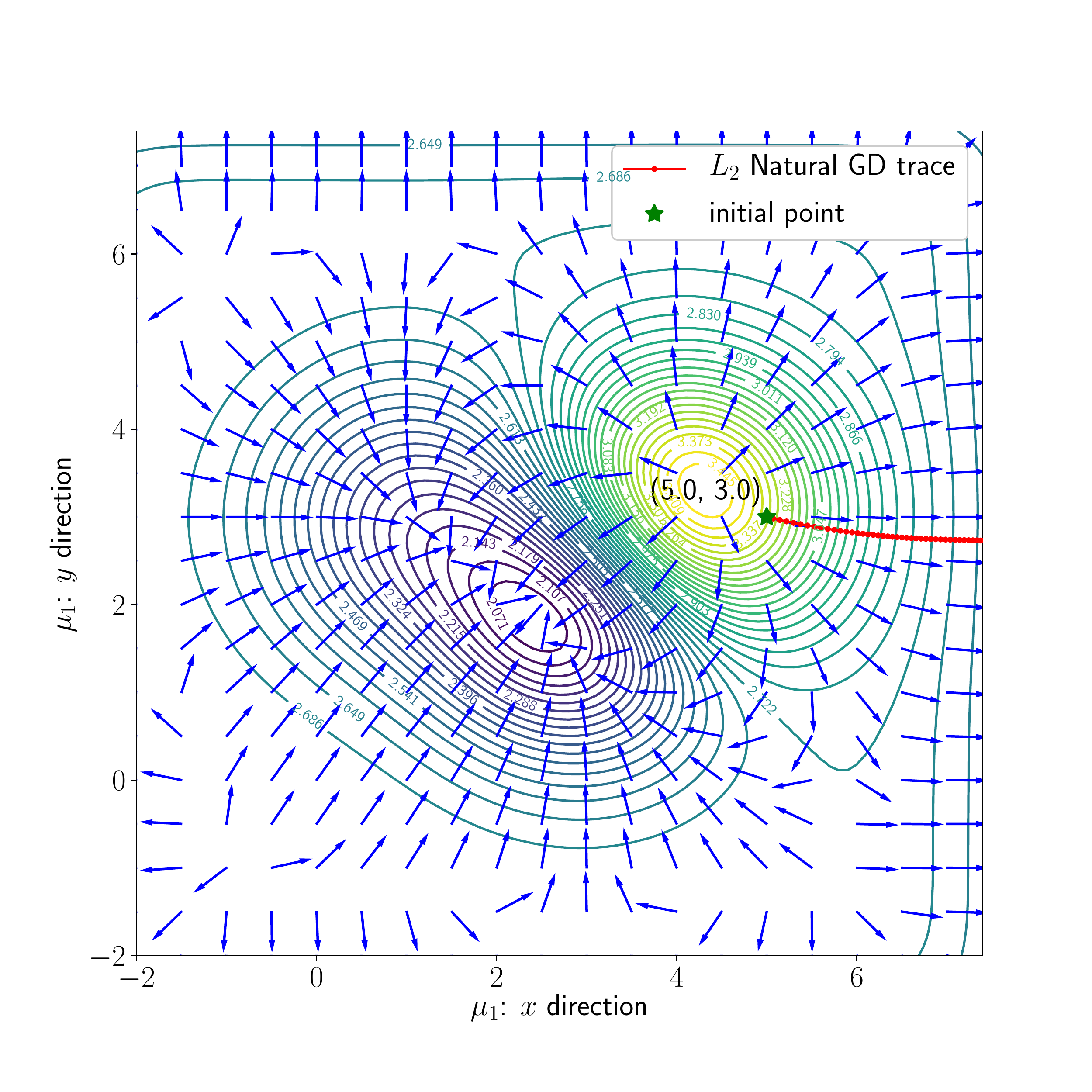}\label{fig:mu1_two_mixture_l2_init1}}
\subfloat[FR NGD]{\includegraphics[width = 0.166\textwidth]{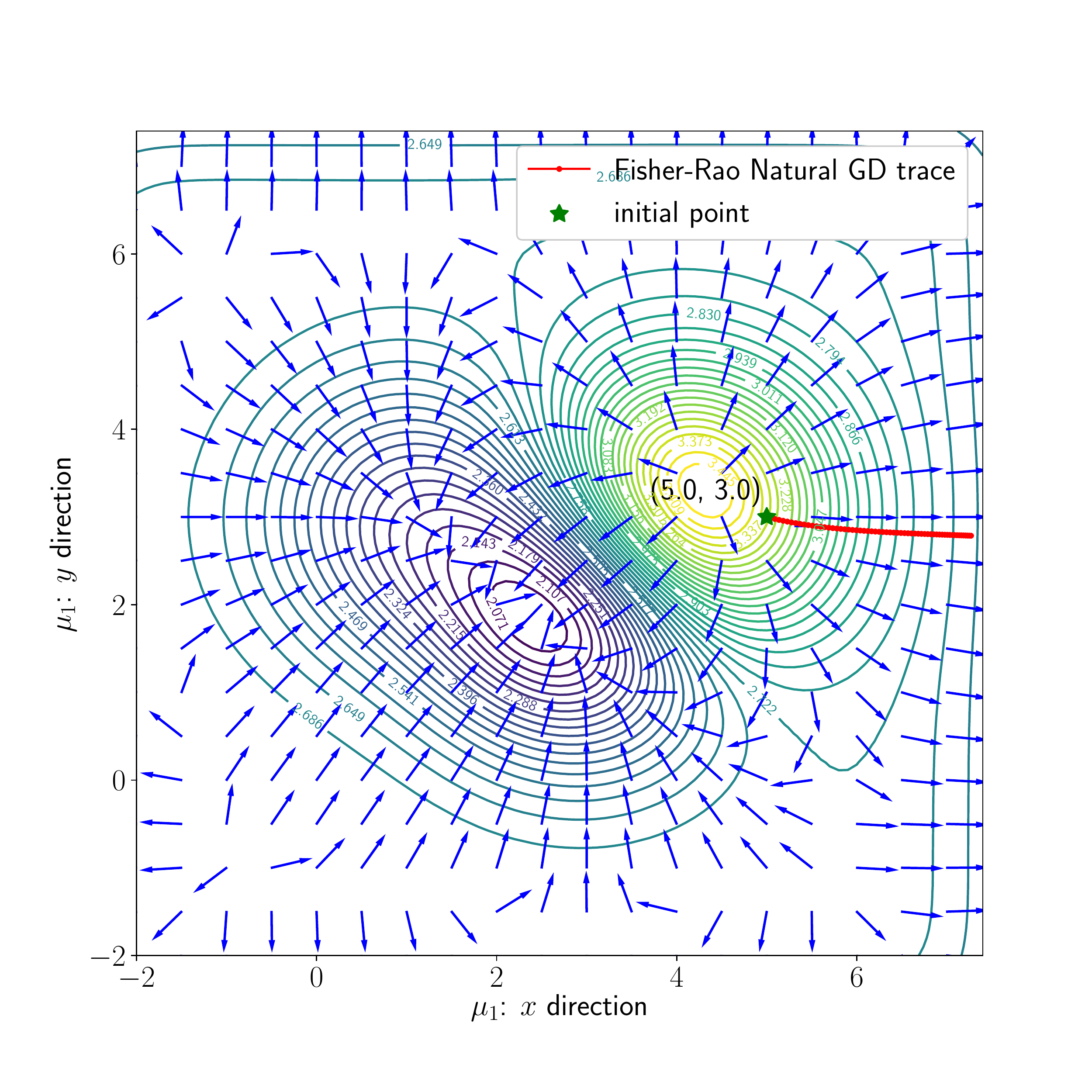}\label{fig:mu1_two_mixture_FR_init1}}
\subfloat[$H^1$ NGD]{\includegraphics[width = 0.166\textwidth]{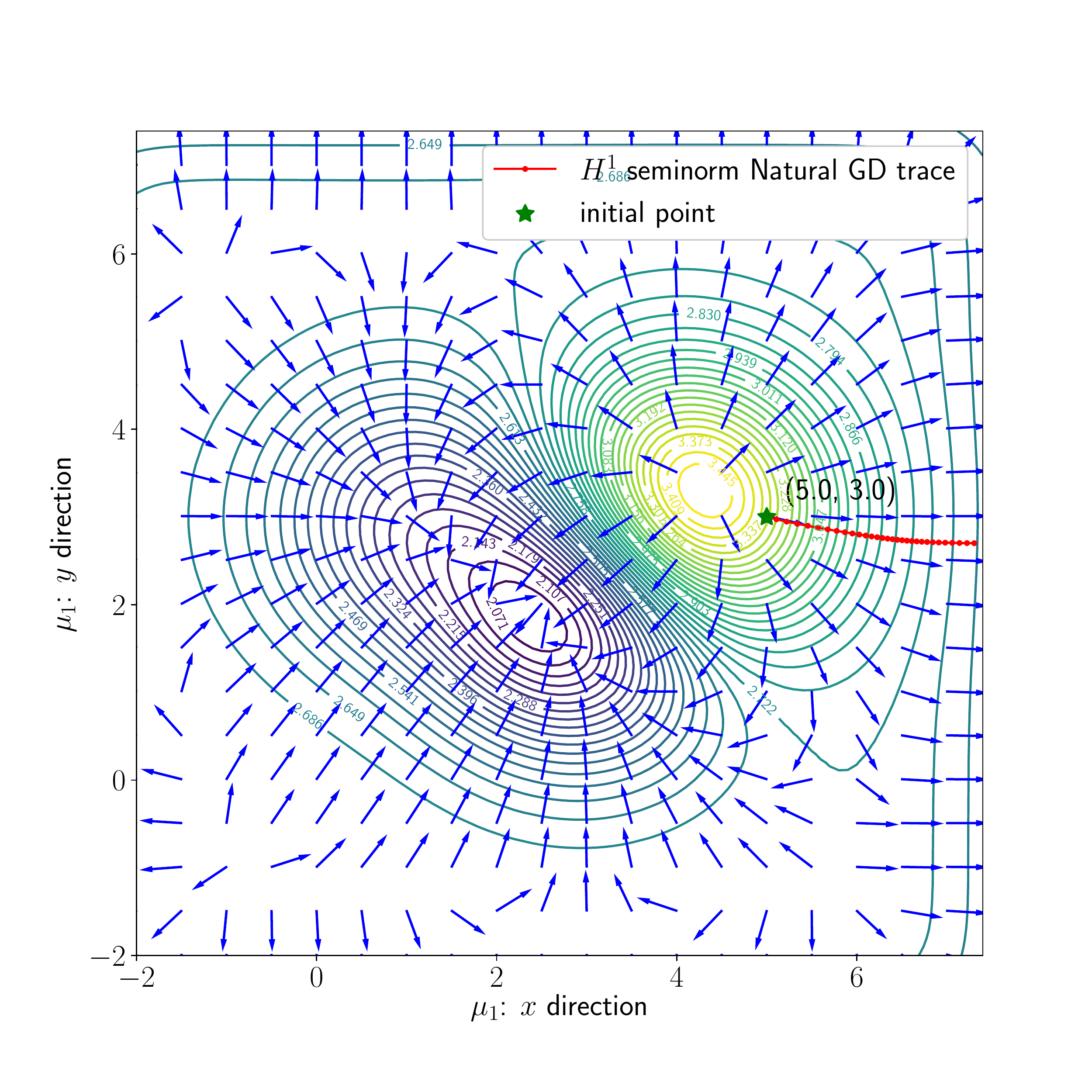}\label{fig:mu1_two_mixture_H1_init1}}
\subfloat[$H^{-1}$ NGD]{\includegraphics[width = 0.166\textwidth]{Second_Setting/H1-mu1-d2-43-lrd2-30-ini53-N71S10Covd6-cen2d25VFtog-eps-converted-to.pdf}\label{fig:mu1_two_mixture_Hinv1_init1}}
\subfloat[$W_2$ NGD]{\includegraphics[width = 0.166\textwidth]{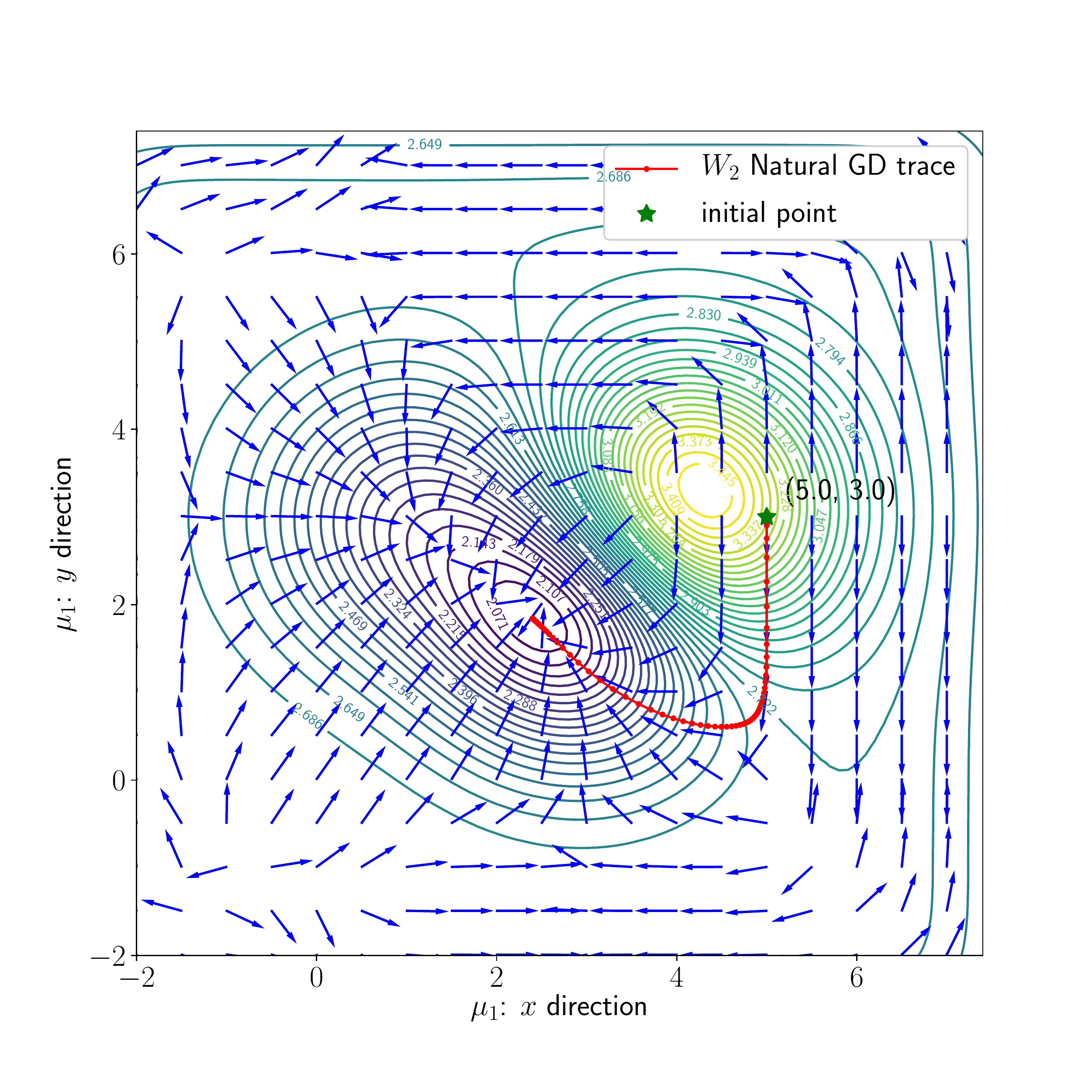}\label{fig:mu1_two_mixture_w2_init1}}
\caption{Gaussian mixture example: level sets, vector fields and convergent paths using GD and different NGD methods to invert $\mu_1$. All algorithms start from initial guess $(5,3)$. }
    \label{fig:two_mixture_setting_two}
\end{figure}

\subsection{Gaussian mixture model}\label{subsec:GaussianMixture}
Consider the Gaussian mixture model, which assumes that all the data points are generated from a mixture of a finite number of normal distributions with unknown parameters. Consider a probability density function $\rho(x;\theta):\R^d\mapsto \R^+$ where
\[ \rho(x;\theta) = w_1 \mathcal{N}(x;\mu_1,\Sigma_1) + \ldots + w_i \mathcal{N}(x;\mu_i,\Sigma_i) +\ldots+ w_k \mathcal{N}(x;\mu_k,\Sigma_k). \]
The $i$-th Gaussian, denoted as $\mathcal{N}(x;\mu_i,\Sigma_i)$ with the mean vector $\mu_i\in\R^d$ and the covariance matrix $\Sigma_i\in \R^{d\times d}$, has a weight factor $w_i\geq 0$. Note that $\sum_i w_i = 1$. Here, $\theta$ could represent parameters such as $\{w_i\}$, $\{\mu_i\}$ and $\{\Sigma_i\}$. We formulate the inverse problem of finding the parameters as a data-fitting problem by minimizing the least-squares loss $f(\rho(\theta))$ on a compact domain $\Omega$ where the objective function follows~\eqref{eq:least-squares-obj}. Here, $\rho^*$ is the observed reference density function. Note that the dependence between the state variable $\rho$ and the parameter $\theta$ is explicit here. Thus, we can compute the Jacobian $\partial_\theta \rho$ analytically, and the numerical scheme follows~\Cref{subsec:Z_avail}.

We consider reference $\rho^*(x) = 0.3\mathcal{N}(x; (1, 3), 0.6 I) + 0.7 \mathcal{N} (x; (3, 2), 0.6I )$ and the domain $\Omega = [-2.75, 7.25]^2$. We fix $\mu_2$ and the weights to be incorrect and invert $\theta = \mu_1$. That is, $\rho(x; \theta) = 0.2 \mathcal{N} (x; \theta, 0.6I) + 0.8 \mathcal{N} (x; (4, 3), 0.6I)$. \Cref{fig:two_mixture_setting_two} shows the convergence paths of GD and  $L^2$, Fisher--Rao, $H^1$, $H^{-1}$, $W_2$ NGD methods under the initial guess $(5,3)$, which is chosen since it belongs to different basins of attractions for different optimization methods. We choose the largest possible step size such that the objective function monotonically decays. They are $0.3$, $0.04$, $0.8$, $0.2$, $0.2$ and $3$ for methods in~\Cref{fig:two_mixture_setting_two} from left to right. WNGD converges to the global minimum while all other methods converge to local minima by taking different convergence paths.

\begin{figure}
\centering
\subfloat[Standard gradient descent]{\includegraphics[width = 0.33\textwidth]{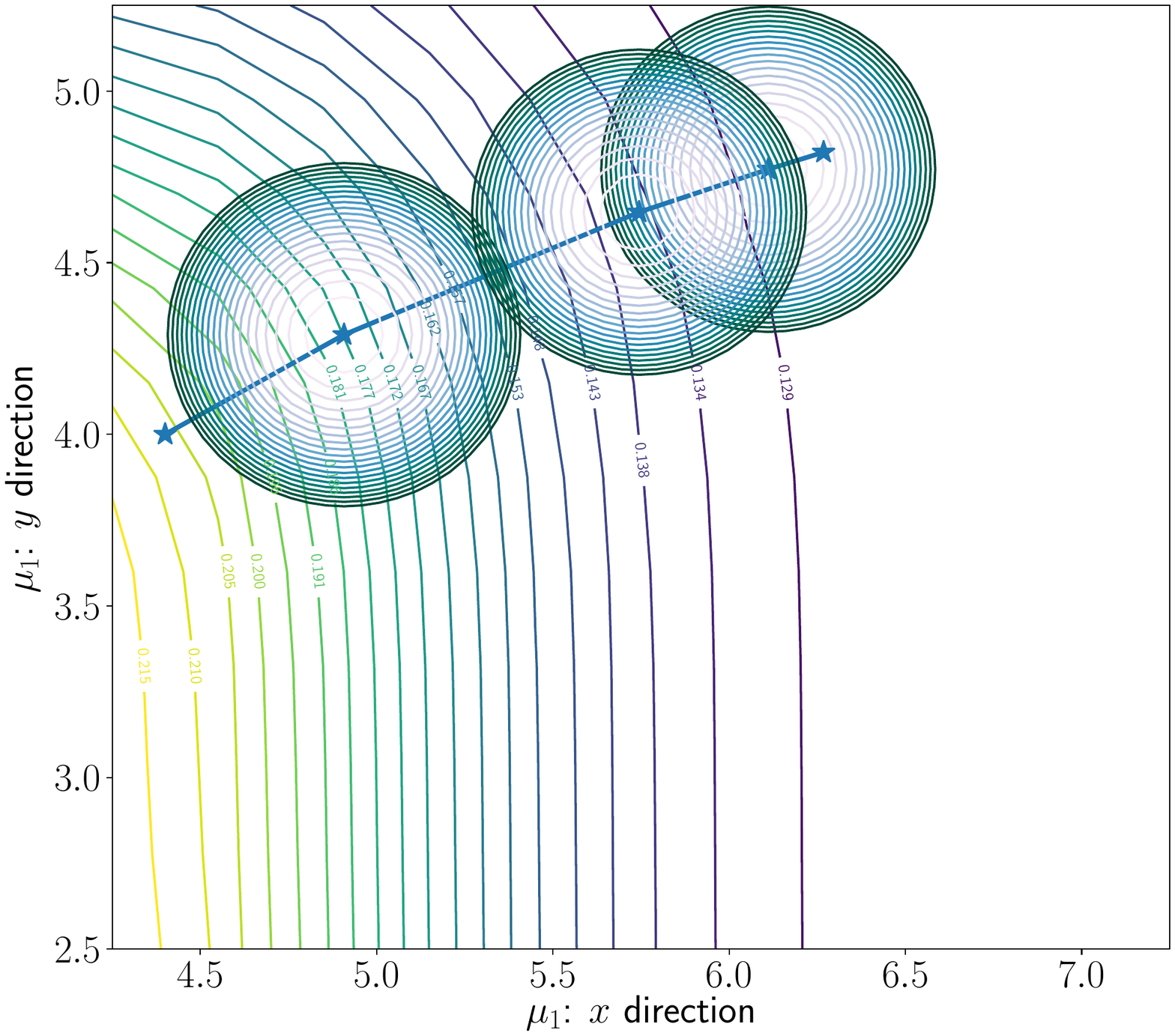}\label{fig:mu1_two_mixture_std_vecfields}}
\subfloat[$L^2$ natural gradient]{\includegraphics[width = 0.33\textwidth]{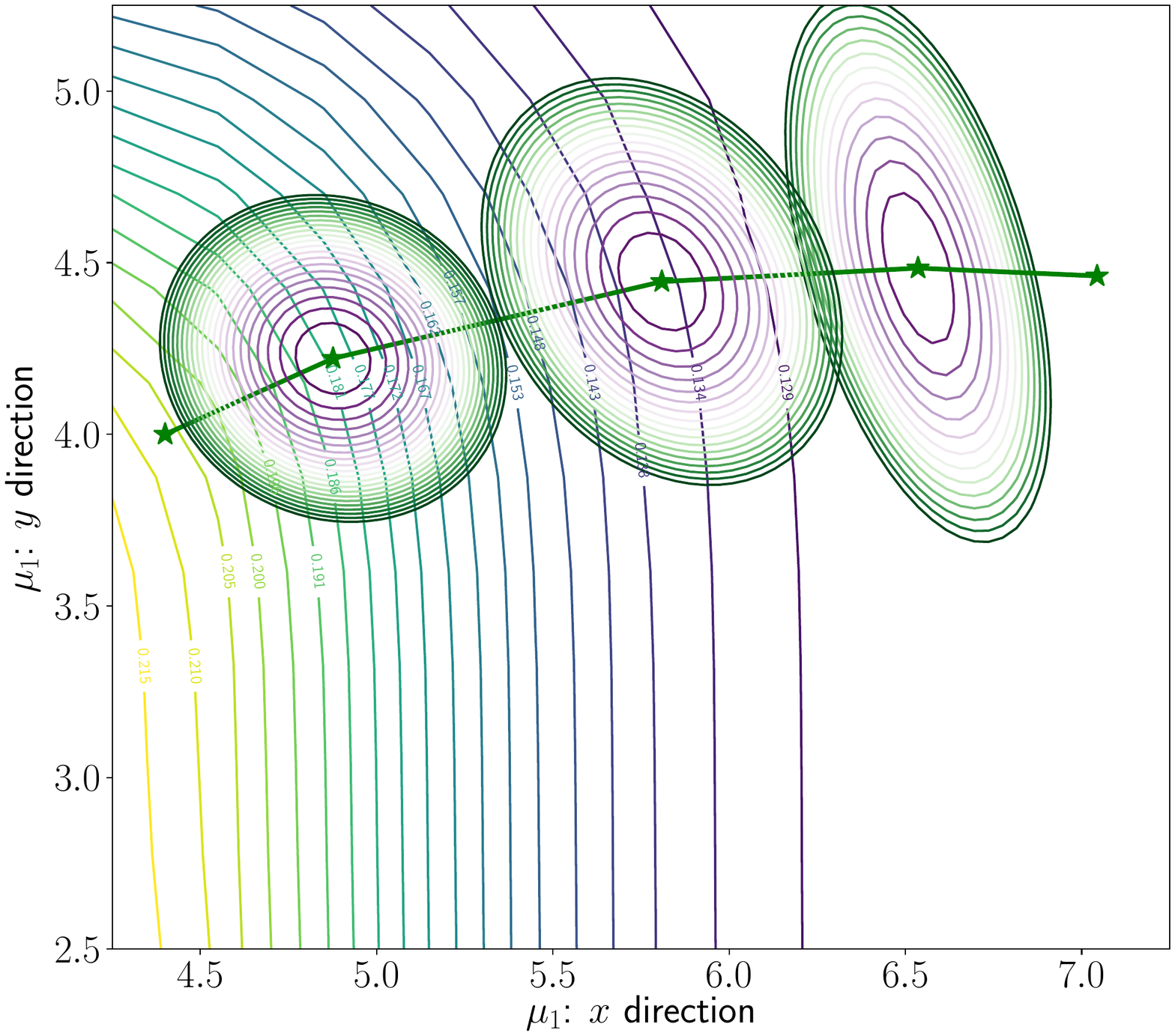}\label{fig:mu1_two_mixture_l2_vecfield}}
\subfloat[$W_2$ natural gradient]{\includegraphics[width = 0.33\textwidth]{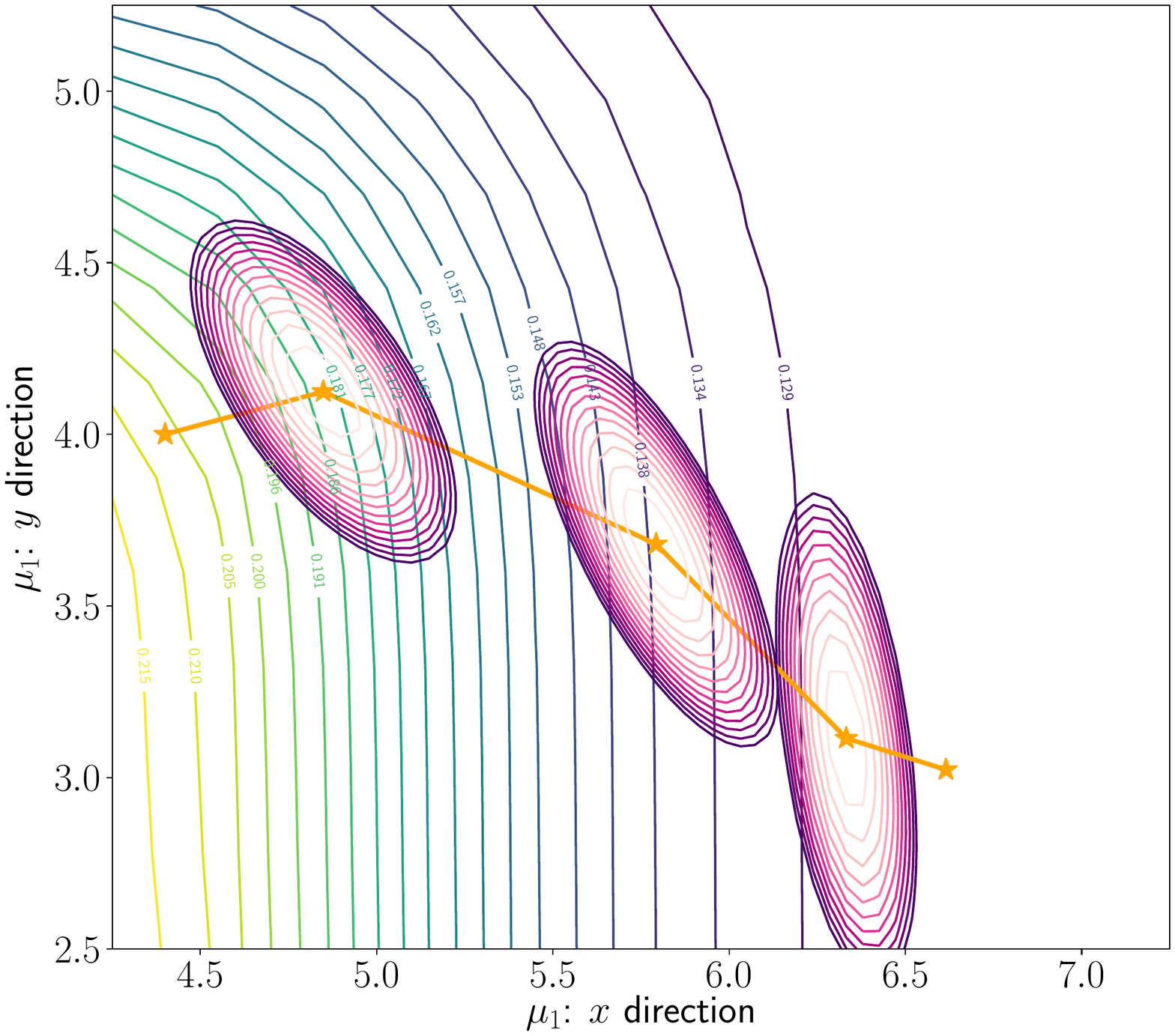}
\label{fig:mu1_two_mixture_w2_vecfield}}
\caption{The local quadratic models of GD, $L^2$ NGD and $W_2$ NGD in the first several iterations.}
\label{fig:mu1_two_mixture_local_quad}
\end{figure}


We aim to gain better understanding regarding their different convergence behaviors. Given a fixed $l$-th iterate, different algorithms find the $(l+1)$-th iterate, but based on different ``principles'' nicely revealed in the proximal operators~\eqref{eq:nat_proximal} and~\eqref{eq:std_proximal}. Here, we use $\theta_{\text{std}}^{l+1} $, $  \theta_{W_2}^{l+1}$, and $\theta_{L^2}^{l+1}$ to denote the next iterates based on GD, $L^2$ NGD and WNGD, respectively. We then have
\begin{align*}
\theta_{\text{std}}^{l+1} &=
     \theta^l+\argmin_{h} \bigg\{  \nabla_\theta f^\top h + \frac{1}{2\tau} h^\top  h \bigg\} \approx \argmin_{\theta} \bigg\{ f(\rho(\theta)) + \frac{  { |\theta-\theta^l|^2}}{2\tau} \bigg\},\\
\theta_{L^2}^{l+1} &=  \theta^{l} +    \argmin\limits_h\bigg\{  \nabla_\theta f^\top h + \frac{1}{2\tau} h^\top  {\partial_\theta \rho ^\top \partial_\theta \rho }~h \bigg\}  \approx \argmin\limits_\theta \bigg\{ f(\rho(\theta)) + \frac{ { ||\rho(\theta)-\rho(\theta^{l}) ||_2^2 } }{2\tau} \bigg\} , \\
  \theta_{W_2}^{l+1}&=  
  \theta^{l} +   \argmin\limits_h\bigg\{  \nabla_\theta f^\top h + \frac{1}{2\tau} h^\top {  (B^\dagger \partial_\theta \rho) ^\top B^\dagger \partial_\theta \rho  }\ h \bigg\} \approx \argmin\limits_\theta \bigg\{ f(\rho(\theta)) + \frac{  { W_2^2(\rho(\theta) ,\rho(\theta^{l}) ) }}{2\tau} \bigg\} .
\end{align*}
The above equations show that, locally, different (N)GD methods solve different quadratic problems given the same step size $\tau$. In~\Cref{fig:mu1_two_mixture_local_quad}, we illustrate the level set of each quadratic problem for which the minimum is selected as the next iterate. The level set of the same objective function $f(\rho(\theta))$ is shown in the background. Our observation aligns with the example in~\cite[Fig.~3]{chen2020optimal}.

\subsection{Physics informed neural networks}\label{sec:PINN}
Physics-informed neural networks (PINN) is a variational approach to solve PDEs with the solution parameterized by neural networks~\cite{raissi2019physics}. Here, as an example, we use PINN to solve the 2D Poisson equation on the domain $\Omega = [-1,1]^2$,
\[
        -\Laplace u = \phi, \quad \text{with } u =  \psi \,\, \text{on } \partial \Omega,
\]
where $\phi(x) =  2\pi^2 \sin(\pi x_1) \sin(\pi x_2) + 18\pi^2 \sin(3\pi x_1) \sin(3\pi x_2)$ and $\psi(x) = 3$,  whose solution is $u(x) = \sin(\pi x_1) \sin(\pi x_2) + \sin(3\pi x_1) \sin(3\pi x_2)+ 3$, $x = [x_1,x_2]^\top$. The training loss function is
\[
f(\rho(\theta)) =   \frac{\gamma}{N_1}\sum_{i=1}^{N_1} |  \Laplace \rho(x_i,\theta) + \phi(x_i)|^2  +  \frac{2-\gamma}{N_2} \sum_{j=1}^{N_2} |\rho(x_j,\theta) -\psi(x_j)|^2,
\]
where $\rho(x,\theta)$ is a feed-forward neural network of shape $(2, 20, 30, 20, 1)$ with the hyperbolic tangent \texttt{tanh} as the activation function.  The parameters are the weights and biases, denoted by $\theta$. We use $N_1 = 2304$ collocation points in the domain interior and $N_2 = 196$ points on $\partial \Omega$, both equally spaced. We set $\gamma = 0.01$ to balance the two terms in the loss function. For a weight matrix of size  $d_1$-by-$d_2$, we initialize its entries i.i.d.~following the normal distribution $\mathcal{N}(0, \frac{2}{d_1 + d_2})$.  All biases are initialized as zero, except the one in the last layer, which is set to be $3$. 
We fix the random seed to ensure the same initialization for all optimization algorithms of interests. 

We train PINN using GD and different NGDs based on metrics discussed in~\Cref{sec:math_nat}. We use back-tracking line search to select the step size (learning rate) in (N)GD algorithms. The true solution is shown in~\Cref{fig:PINN u}, while~\Cref{fig:PINN loss vs iter,fig:PINN loss vs time} show the loss value decay with respect to the number of iterations and the wall clock time, respectively. We can see that all NGD methods are faster than GD, while $H^1$ and $\dot{H}^1$-based NGDs yield the fastest convergence in both comparisons. Neural networks can suffer from slow convergence on the high-frequency parts of the residual due to its intrinsic low-frequency bias~\cite{yu2022quadrature}. The $H^1$/$\dot{H}^1$-based NGDs enforce extra weights on the oscillatory components of the Jacobian, giving faster convergence than $L^2$ NGD. In contrast, $H^{-1}$/$\dot{H}^{-1}$ NGDs bias towards the smooth components of the Jacobian, which delay the convergence of high-frequency residuals and thus the overall convergence. As discussed in~\Cref{rmk:cost_discuss}, WNGD requires a $\rho$-dependent matrix $L$, which increases the wall clock time per iteration. Interestingly, when the loss value becomes small, WNGD has a faster decay rate than $H^{-1}/\dot{H}^{-1}$ NGDs despite being asymptotically equivalent in spectral properties (see~\Cref{rmk:Hm1_vs_W2}), demonstrating the potential benefits of having a state-dependent information matrix $G(\theta)$.

\begin{figure}
\centering
\subfloat[True solution]{\includegraphics[width = 0.33\textwidth]{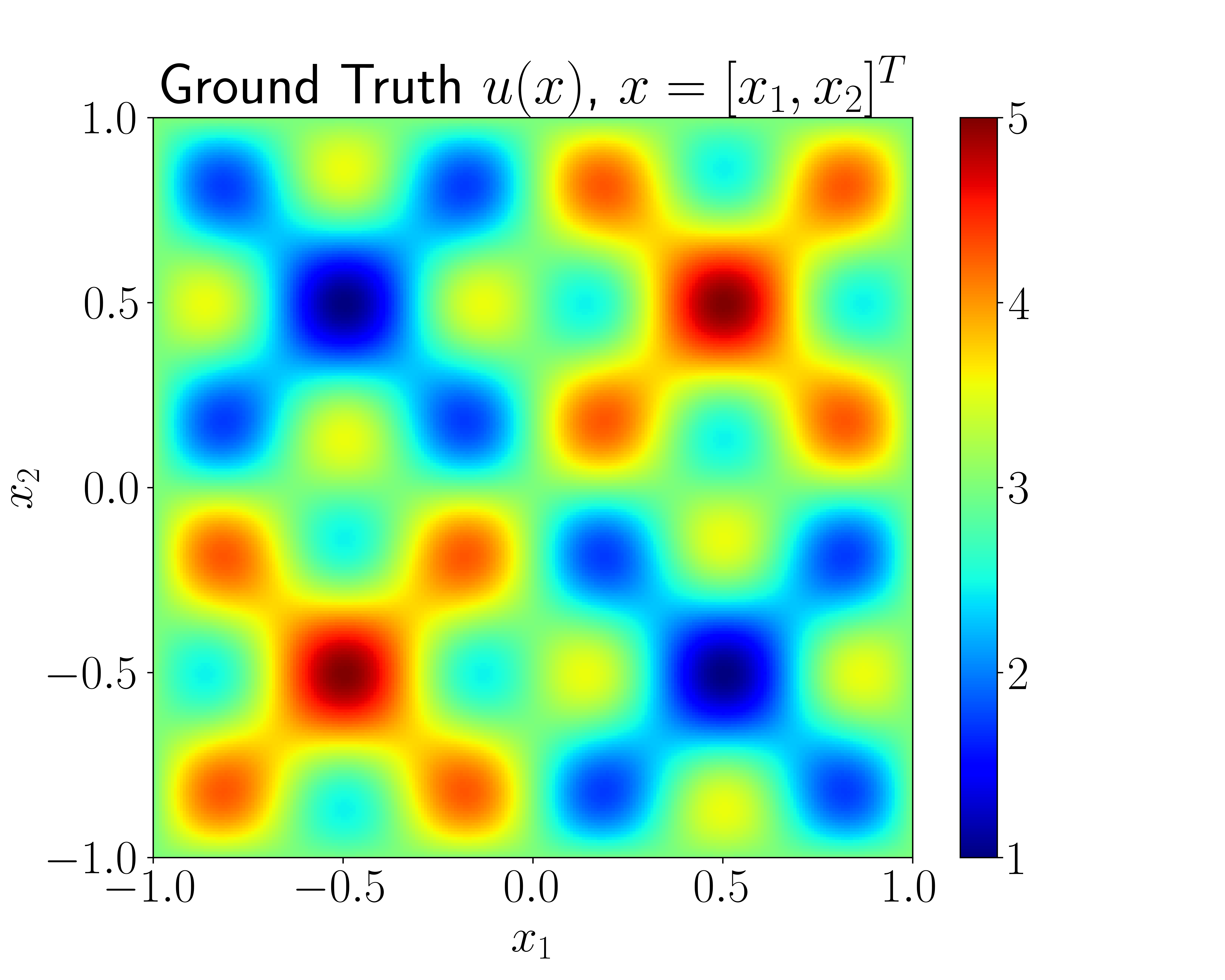}\label{fig:PINN u}}
\subfloat[Loss decay vs.~iteration number]{\includegraphics[width = 0.33\textwidth]{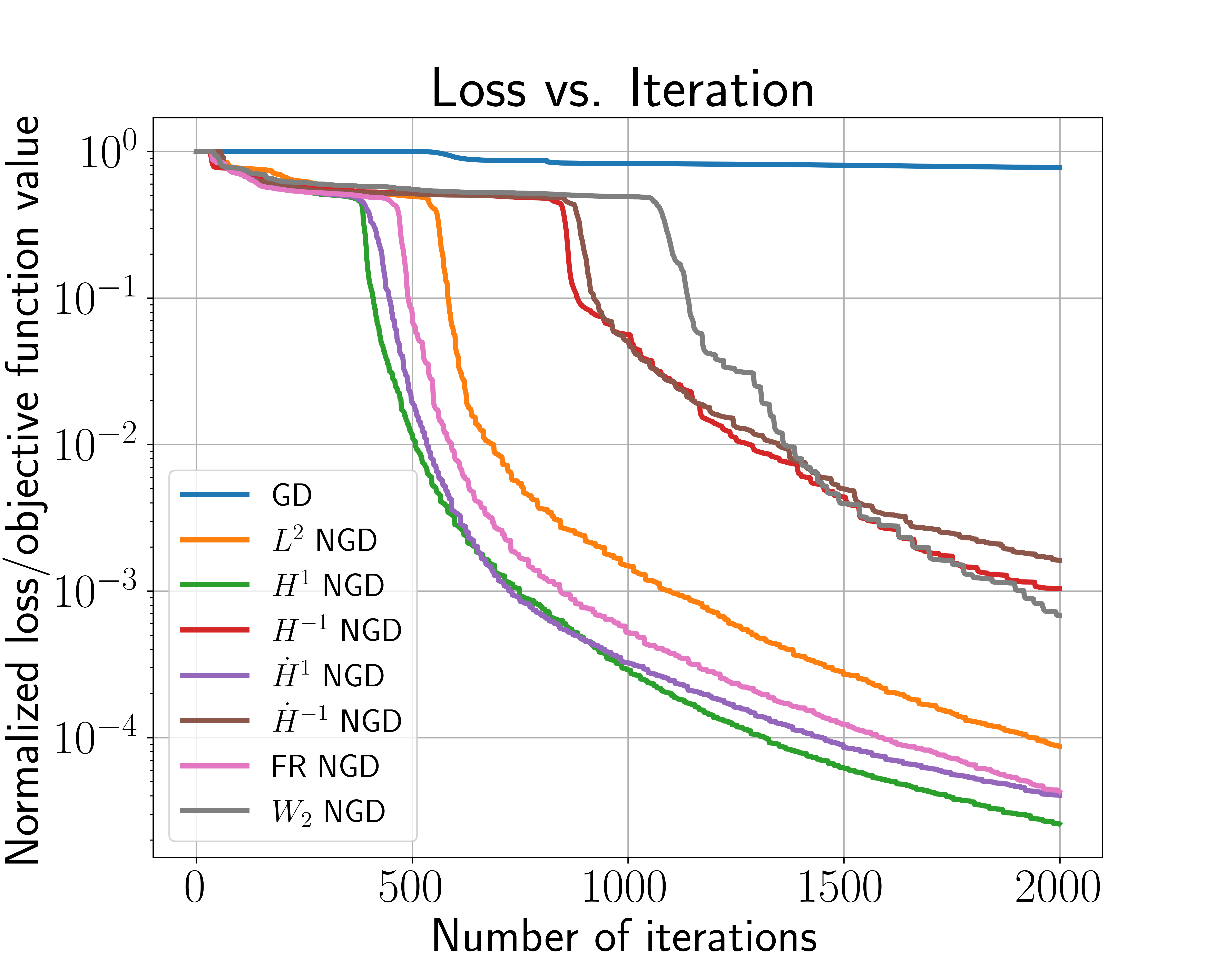}
\label{fig:PINN loss vs iter}}
\subfloat[Loss decay vs.~wall clock time]{\includegraphics[width = 0.33\textwidth]{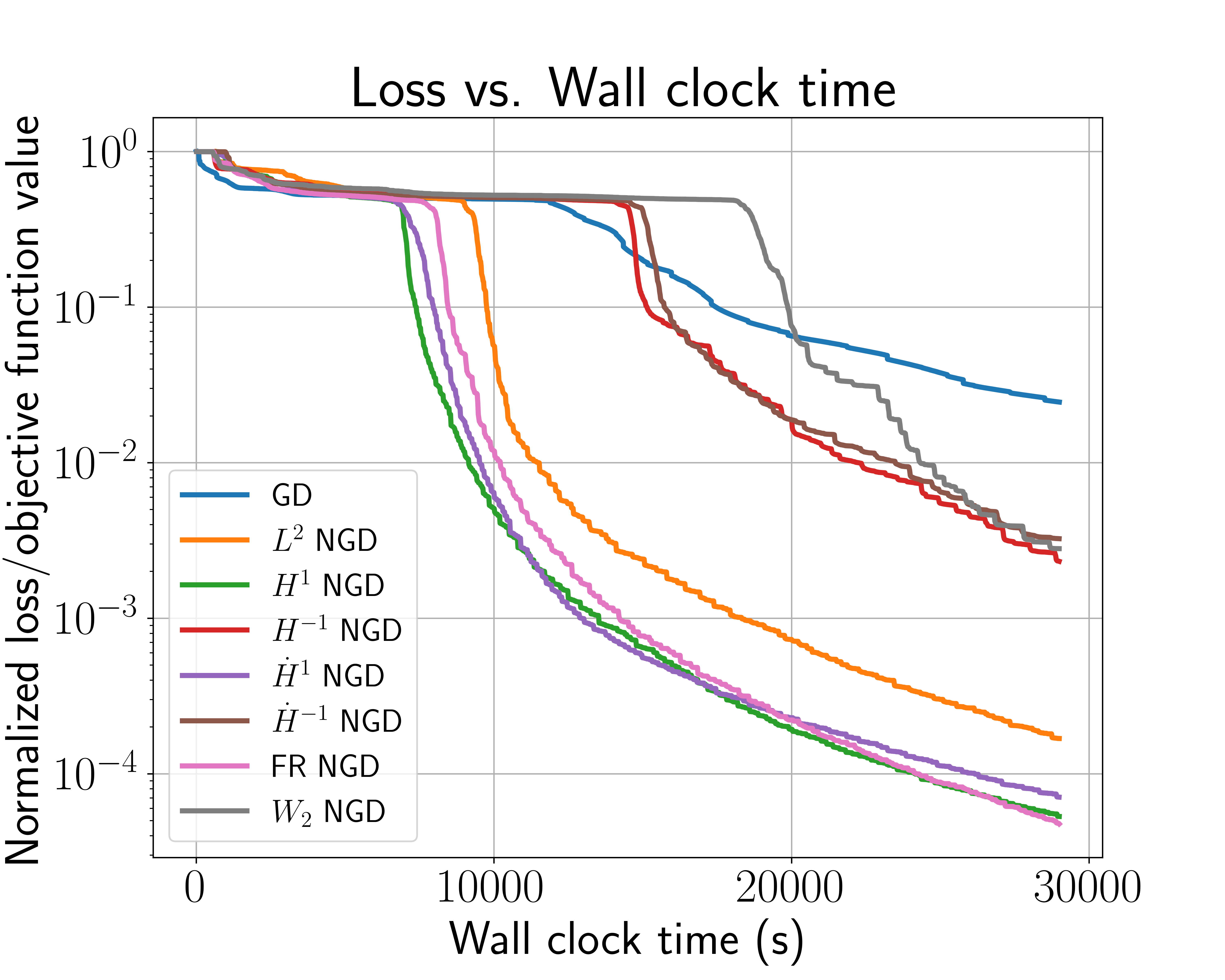}\label{fig:PINN loss vs time}}
\caption{(a): PINN example true solution; (b) loss function value decay in terms of the number of iterations; (c) loss function value decay in terms of the wall clock time.\label{fig:PINN}}
\end{figure}

\subsection{Full waveform inversion} \label{subsec:FWI}
Finally, we present a full waveform inversion (FWI) example where the Jacobian is not explicitly given. As a PDE-constrained optimization, the dependence between the data and the parameter is implicitly given through the scalar wave equation
\begin{equation}\label{eq:wave}
m(x) u_{tt}(x,t) + \Laplace u (x,t) = s (x,t),
\end{equation}
where $s(x,t)$ is the source term and~\eqref{eq:wave} is equipped with the initial condition $u(x,0) = u_t(x,0) = 0$ and an absorbing boundary condition to mimic the unbounded domain. 


After discretization, the unknown function $m(x)$ becomes a finite number of unknowns, which we denote by $\theta$ for consistency. Unlike the Gaussian mixture model, the size of $\theta$ in this example is large as $p=36720$. We obtain the observed data $\rho_{r} = u(x_r,t)$ at a sequence of receivers $\{x_r\}$, for $r=1,\ldots, n_r$. The least-squares objective function is
\begin{equation}\label{eq:test2_obj}
    f(\rho(\theta)) = \frac{1}{2} \sum_{i=1}^{n_s}\sum_{r=1}^{n_r} \|\rho^*_{i,r} - \rho_{i,r}(\theta)\|_2^2, 
\end{equation}
where $\rho^*$ is the observed reference data, and $i$ is the source term index to consider inversions with multiple sources $\{s_{i}(x,t)\}$ as the right-hand side in~\eqref{eq:wave}. In our test, $n_s = 21$ and $n_r = 306$. 

The true parameter is presented in~\Cref{fig:fwi_true}.
We remark that minimizing~\eqref{eq:test2_obj} with the constraint~\eqref{eq:wave} is a highly nonconvex problem~\cite{virieux2009overview}. We avoid dealing with the nonconvexity by choosing a good initial guess; see~\Cref{fig:fwi_init}. One may also use other objective functions such as the Wasserstein metric to improve the optimization landscape~\cite{engquist2020optimal}. We follow~\Cref{subsec:Z_unavail} to carry out the implementation for various NGD methods since the Jacobian $\partial_\theta \rho$ is not explicitly given, and the adjoint-state method has to be applied based on~\eqref{eq:wave}. The step size is chosen based on back-tracking linear search. We use the same criteria for all algorithms.
The GD (see~\Cref{fig:fwi_std}) converges slowly compared to the NGD methods, while $\dot{H}^1$, $L^2$, $\dot{H}^{-1}$ and $W_2$ NGDs are in descending order in terms of image resolution measured by both the objective function and the structural similarity index measure (SSIM); see~\Cref{fig:fwi_l2}-\ref{fig:fwi_loss}. The convergence history in~\Cref{fig:fwi_loss} shows the objective function decay with respect to the number of propagations (see~\Cref{tab:propagation number}). For FWI, each propagation corresponds to one wave equation (PDE) solve with different source terms. Note that wavefields are not naturally probability distributions. Thus, when we implement the $W_2$ natural gradient, we normalize the data to be probability densities following~\cite{engquist2019seismic,engquist2020optimal}. As we have discussed in~\Cref{rmk:Hm1_vs_W2}, the $W_2$ and $\dot{H}^{-1}$ natural gradients are closely related, which are also reflected in this numerical example as the reconstructions in~\Cref{fig:fwi_Hm1,fig:fwi_W2} are very similar. All the tests shown in~\Cref{fig:fwi} directly demonstrate that NGDs are typically faster than GD, and more importantly, the choice of the metric space $(\M, g)$ for NGD (see~\eqref{eq:nat_grad_gen}) also has a direct impact on the convergence rate.


\begin{figure}
\centering
\subfloat[true parameter]{\includegraphics[width = 0.25\textwidth]{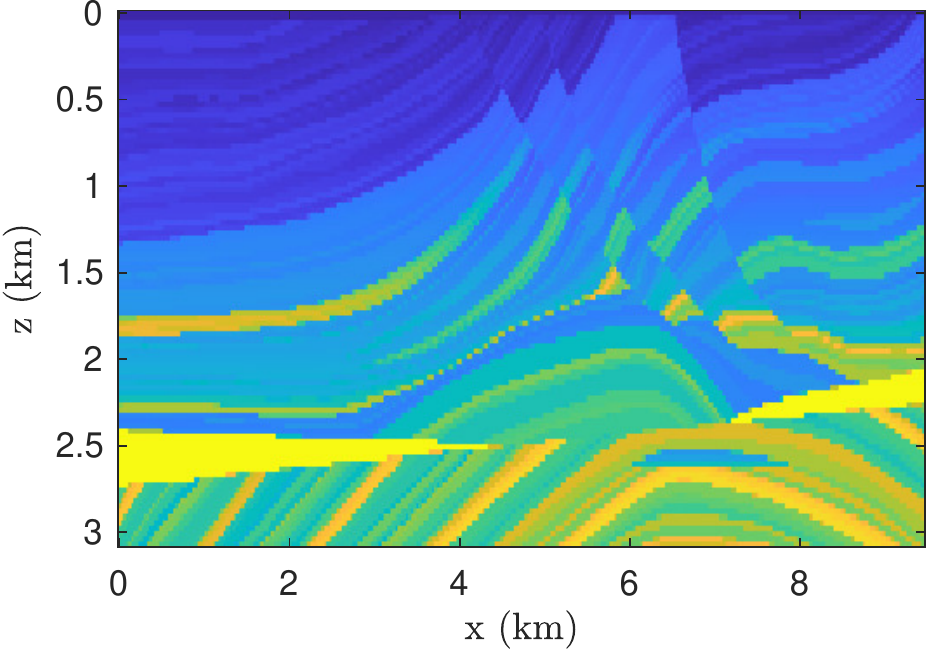}\label{fig:fwi_true}}
\subfloat[initial, SSIM$=0.31$]{\includegraphics[width = 0.25\textwidth]{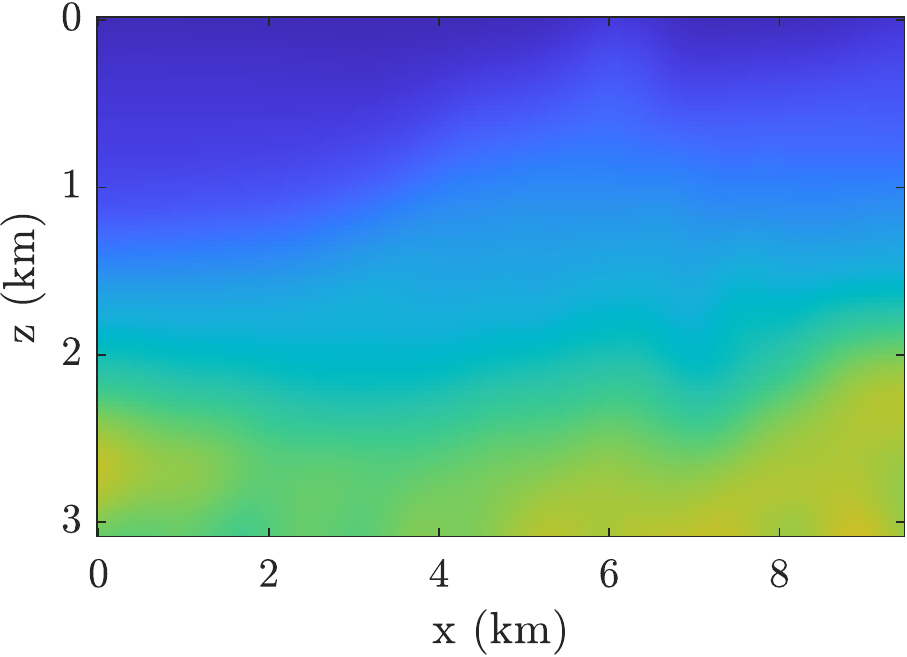}\label{fig:fwi_init}} 
\subfloat[GD, SSIM$=0.44$]{\includegraphics[width = 0.25\textwidth]{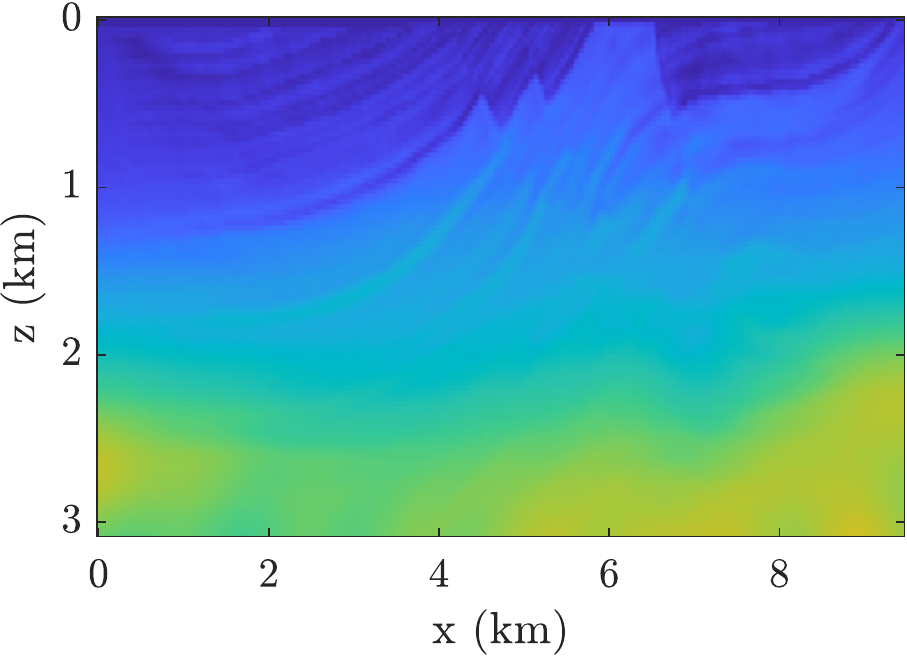}\label{fig:fwi_std}}
\subfloat[$L^2$ NGD, SSIM$=0.58$]{\includegraphics[width = 0.25\textwidth]{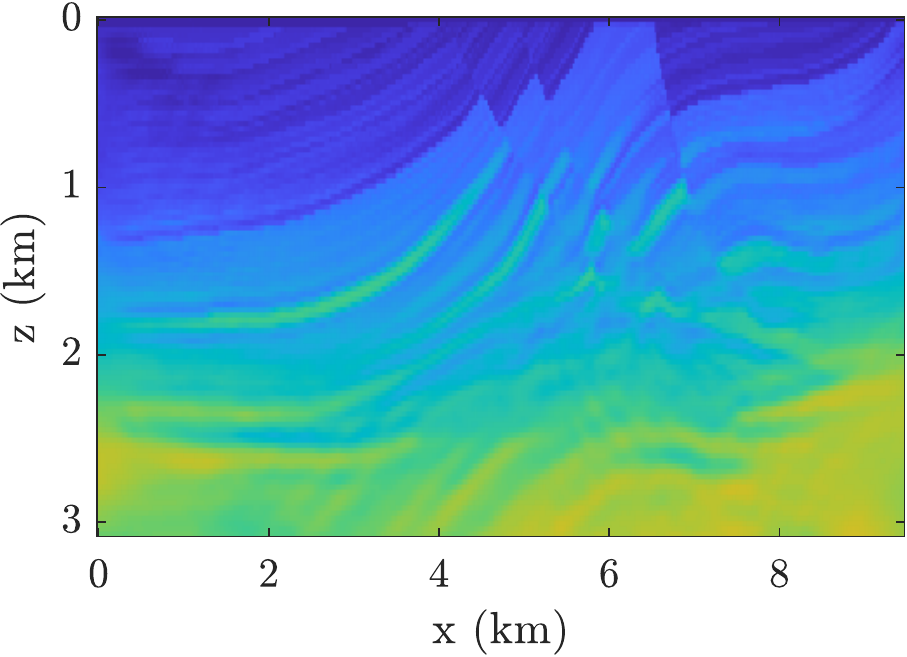}\label{fig:fwi_l2}}\\
\subfloat[$\dot{H}^{-1}$ NGD, SSIM$=0.53$]{\includegraphics[width = 0.25\textwidth]{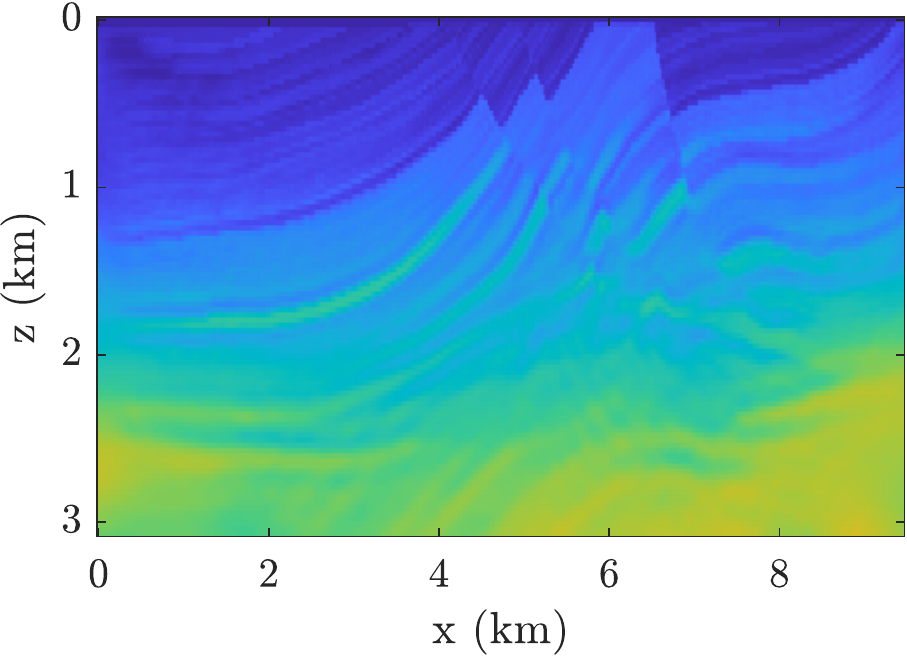}\label{fig:fwi_Hm1}}
\subfloat[$W_2$ NGD, SSIM$=0.53$]{\includegraphics[width = 0.25\textwidth]{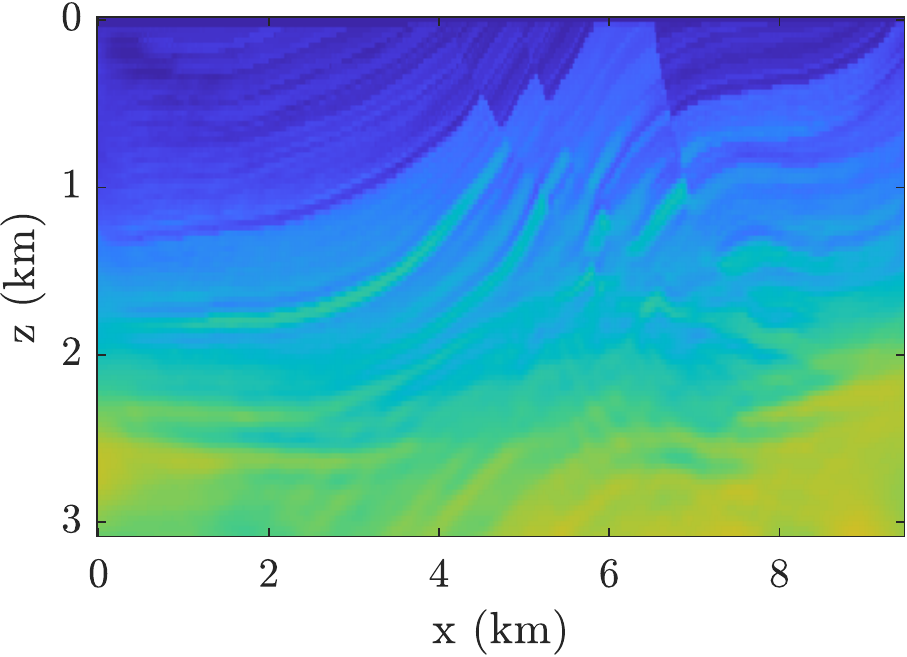}\label{fig:fwi_W2}}
\subfloat[$\dot{H}^{1}$ NGD, SSIM$=0.61$]{\includegraphics[width = 0.25\textwidth]{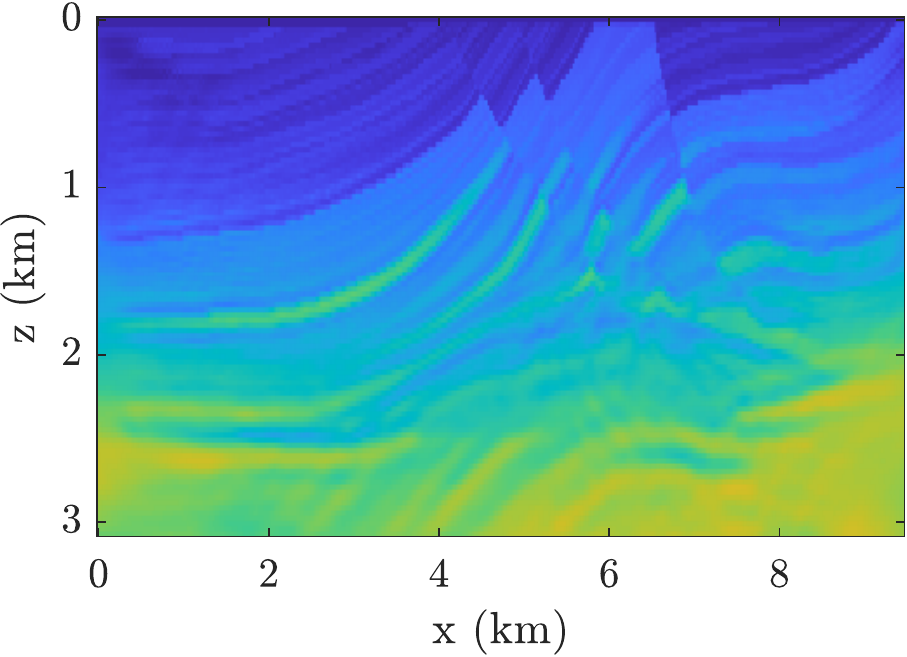}\label{fig:fwi_H1}}
\subfloat[convergence history]{\includegraphics[width = 0.25\textwidth]{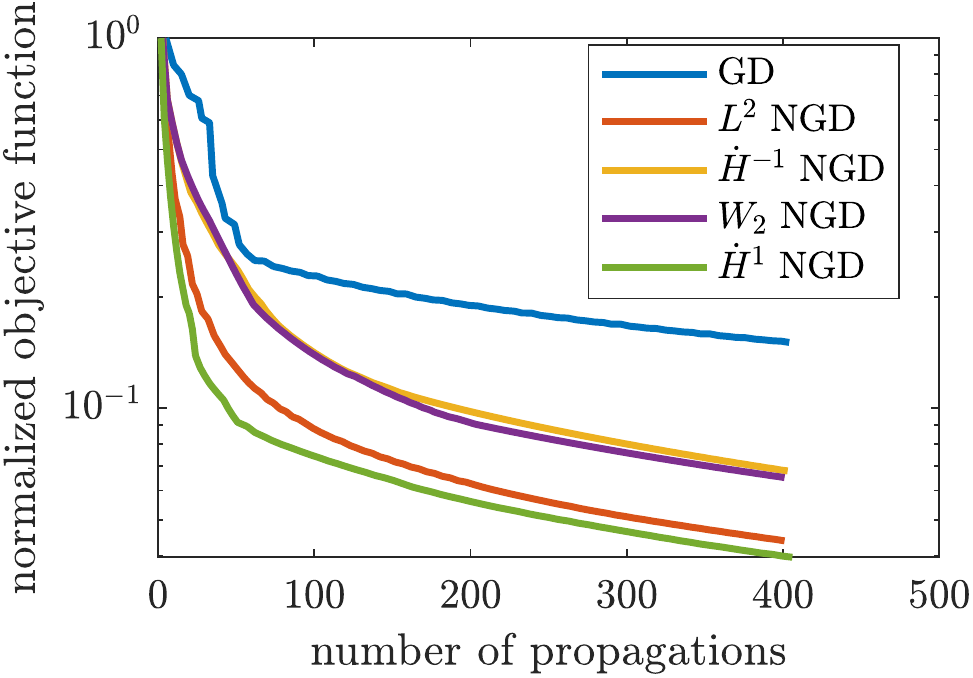}\label{fig:fwi_loss}}
\caption{FWI example: (a)~ground truth; (b)~initial guess; (c)-(g) inversion results using GD and NGDs based on the $L^2$,  $\dot{H}^{-1}$, $W_2$ and $\dot{H}^{1}$ metrics after $400$ PDE solves; (h) the history of the objective function decay versus the number of propagations/PDE solves. SSIM denotes the structural similarity index measure compared with~(a). A bigger value means better similarity.} \label{fig:fwi}
\end{figure}

\section{Conclusions}\label{sec:conclusions}

Inspired by the natural gradient descent (NGD) method in learning theory, we develop efficient computational techniques for PDE-based optimization problems for generic choices of the \textit{``natural''} metric. NGD   exploits the geometric properties of the state space, which is particularly appealing for PDE applications that have rich flexibility in choosing the metric spaces.

Handling the high-dimensional parameter space and state space are the two main computational challenges of NGD methods. Here, we propose numerical schemes to tackle the high-dimensional parameter space when the forward model, with a relatively low-dimensional state space, is discretized on a regular grid. 
Our approach relies on reformulating the problem of finding NGD directions as standard $L^2$-based least-squares problems on the continuous level. After discretization, the NGD directions can be efficiently computed by numerical linear algebra techniques. We discuss both explicit and implicit forward models by taking advantage of the adjoint-state method.

The second computational challenge of high-dimensional state space stands out for Sobolev and Wasserstein NGDs. In this work, we apply finite differences on regular grids for low-dimensional state space. On the one hand, when the state-space dimension is high, discretization on a regular grid suffers from the curse of dimensionality, and other parameterizations have to be considered. On the other hand, when the state variable is not given on a regular grid, there are other ways to discretize those differential operators, which require more careful attention. For example, generative models are push-forward mappings, representing probability measures in high-dimensional state spaces by point clouds (samples). Applying the Sobolev and Wasserstein NGDs to state variables in the form of empirical distributions will most likely require alternative discretization approaches for differential operators, such as graph- or neural network-based methods.

A very interesting question is what the best ``natural'' metric in NGD should be. Regarding this, we numerically investigated the convergence behaviors of GD and various NGD methods based on different metric spaces. The empirical results indicate that the choice of the metric space in an NGD not only can change the rate of convergence but also influence the stationary point where the iterates converge, given a nonconvex optimization landscape. A rigorous understanding of the ``best'' metric choice for a given problem is an important research direction. For maximum likelihood estimation problems, the Fisher--Rao NGD is asymptotically Fisher-efficient; Sobolev NGDs (e.g., $H^1$ and $\dot{H}^1$) are suitable for solving optimal transport and mean-field game problems~\cite{jacobs2019solving,jacobs2021backforth,liu2021computational,liu2021splitting}; when the metric is induced by $f$ and suitable conditions are met, the corresponding NGD is asymptotically Newton's method~\cite{mallasto2019formalization,chen2020optimal,martens2020new}. Despite these results, to our knowledge, there is no general framework for a systematic derivation of the best natural gradient metric for a given problem.

It is reasonable to believe that as the topic matures, there will be an increasing necessity for efficient techniques for computing NGD directions for a diverse set of problems and metrics. Hence, in this paper, we choose to focus on a \textit{generic computational framework} leveraging state-of-the-art optimization techniques. 
Nonetheless, the geometric formalism considered here could be beneficial for the theoretical understanding of the ``best'' metric choice. Indeed, as mentioned in~\cite[Sec.~15]{martens2020new}, local approximation of the loss function cannot explain all global properties of NGD. The metric in the $\rho$-space, on the other hand, can impact the global properties of $f$. More specifically, it might convexify $f$~\cite[Appendix B]{gangbo22global} or make it Lipschitz, paving a way towards the analysis of the NGD as a first-order method in the $\rho$-space. We find this line of research an intriguing future direction.

Finally, the full potential of randomized linear algebra techniques remains to be explored. We discuss a mini-batch version of our algorithm in~\Cref{subsubsec:minibatch} and several low-rank approximation techniques in~\Cref{sec:qr_low_rank,subsec:Z_unavail_Hut,sec:column_Z}. Nevertheless, the success of randomized linear algebra techniques for very high-dimensional problems warrants a more thorough investigation of the theoretical and computational aspects of these techniques adapted to our setting.

\section*{Acknowledgments}
L.~Nurbekyan was partially supported by AFOSR MURI FA 9550 18-1-0502 grant. W.~Lei was partially supported by the 2021 Summer Undergraduate Research Experience (SURE) at the Department of Mathematics, Courant Institute of Mathematical Sciences, New York University. Y.~Yang was partially supported by the National Science Foundation under Award Number DMS-1913129. This work was done in part while Y. Yang was visiting the Simons Institute for the Theory of Computing in Fall 2021. Y. Yang also acknowledges supports from Dr.~Max R\"ossler, the Walter Haefner Foundation and the ETH Z\"urich Foundation.

\bibliographystyle{siamplain}
\bibliography{natgrad}

\appendix

\section{Symbols and Notations}
See~\Cref{tab:notations} for all the notations in~\Cref{sec:intro,sec:math_nat,sec:num_nat}.
\setlength{\tabcolsep}{10pt} 
\renewcommand{\arraystretch}{1} 
\begin{table}[!ht]
\centering
\caption{Table of notations in~\Cref{sec:intro,sec:math_nat,sec:num_nat}.\label{tab:notations}}
\begin{tabular}{ll} 
\hline
\hline
\textit{\Cref{sec:intro}}\\
\hline
$\theta$ & the unknown parameter \\
$\rho$ & the state variable that depends on $\theta$ \\
$f(\rho)$ & the loss function that depends on $\rho$\\
$(\mathcal{M}, d_\rho)$, $(\Theta, d_\theta)$  & the metric space of $\rho$ and $\theta$, respectively\\
\hline
\hline
\textit{\Cref{sec:math_nat}}\\
\hline
$(\mathcal{M}, g)$  &  the space $\mathcal{M}$ endowed with a Riemannian  metric $g$\\
$T_{\rho} \M$ & the tangent space of $\M$ \\
$p$ & the dimension of the parameter, $\theta \in \Theta \subseteq \mathbb{R}^p$\\
$\partial^g_{\theta_i} \rho(\theta) \in T_{\rho} \M $ & the tangent vector of $\rho(\theta)$ with respect to $\theta_i$ based on \\ & the Riemannian geometry $(\M,g)$, $1\leq i \leq p$\\
$\partial^g_\rho f(\rho) \in T_\rho \M$ & the metric gradient of $f(\rho)$ with respect to $\rho$ based on \\ & the Riemannian geometry $(\M,g)$\\
$\eta^{nat}$, $\eta^{std}$ & the natural and standard gradient directions for $\theta$\\
$\partial_\theta f(\rho(\theta))$ & the gradient of $f(\rho(\theta))$ with respect to $\theta$\\
$P \partial^g_\rho f$ & the $\langle\cdot,\cdot\rangle_{g(\rho)}$-orthogonal projection of $-\partial^g_\rho f$ onto\\
&$\operatorname{span} \{\partial^g_{\theta_1} \rho,\ldots, \partial^g_{\theta_p} \rho \}$\\
$G(\theta)$ & the information matrix $G_{ij}(\theta)=  \langle \partial^g_{\theta_i} \rho, \partial^g_{\theta_j} \rho \rangle_{g(\rho(\theta))}, i,j=1,\ldots,p$ \\
$\zeta, \hat{\zeta}$  & tangent vectors on $T_\rho \M$\\
$\zeta_i = \partial_{\theta_i} \rho$, $i=1,\ldots, p$ &  tangent vectors on the Euclidean space $( L^2(\R^d), \langle \cdot, \cdot \rangle_{L^2(\R^d)})$\\
$\partial_\rho f$ & the metric gradient of $f(\rho)$ in $(L^2(\R^d), \langle \cdot, \cdot \rangle_{L^2(\R^d)})$\\
$G^{L^2}, G^{H^s}, G^{\dot{H}^s}, G^{FR}, G^{W}$ & the information matrices for different Riemannian metrics \\
$\bD^s$ & a differential operator that outputs   a vector of all the \\&  partial derivatives  up to order $s$ where $s\geq0$\\
$A^*$, $A^\dagger$ & the adjoint and the pseudoinverse of the linear operator $A$ \\
$\chi,\hat{\chi}$ & the tangent vectors in $H^{-s}$ mapped from $\zeta, \hat{\zeta}$ in $H^{s}$, $s<0$ \\
$\Laplace$ & the Laplacian operator\\
 $\widetilde{\bD}^s$ & a differential operator that outputs a vector of all the \\ & partial derivatives of positive order up to $s$ where $s> 0$\\
$\Pp_2(\R^d)$ & the set of Borel probability measures of finite second moments \\
$f_\sharp \rho$ & the pushforward distribution of $\rho$ by $f$ \\
$\Gamma(\rho_1,\rho_2)$ &  the set of all  measure $\pi \in \Pp(\R^{2d})$ with  $\rho_1$ and $\rho_2$ as marginals \\
$v,\hat{v},w,\{v_i\}_{i=1}^p $ & the tangent vectors in $T_\rho \Pp_2(\R^d) \subset L^2_\rho(\R^d;\R^d)$ \\
$\{\Tilde v_i\}_{i=1}^p$ & the re-normalized Wasserstein tangent vectors, $\Tilde v_i =  \sqrt{\rho}~v_i$ \\
$\bB$ & the differential operator defined by $\bB  \Tilde{v}  = -\nabla \cdot (\sqrt{\rho(\theta )}~\Tilde{v})$\\
$ \bB_k $ & a generalized version of $\bB$ given by $\bB_k \tilde{v}=-\nabla \cdot (\rho(\theta)^{k} \Tilde{v} )$ \\
$\bL$ & with different choice of $\bL$, all natural gradient directions  can  be \\ &   formulated as $\eta^{nat}  =\argmin_{\eta \in \R^p}  \|(\bL^*)^\dagger \partial_\rho f+\sum_{i=1}^p \eta_i~ \bL \zeta_i \|^2_{L^2(\R^d)}$ \\
\hline
\hline
\textit{\Cref{sec:num_nat}} \\
\hline
$\rho \in \R^k$ & the discretized state variable \\
$\partial_\rho f$, $Z =\partial_\theta \rho  $ & the finite-dimensional gradient and Jacobian in Euclidean space \\
$L$ & the discretization of the operator $\bL$ for different metric spaces \\
$G_L = Y^\top Y$ & the discretized information matrix, $Y = L Z$\\
$\eta^{nat}_L$ & the natural gradient direction in a unified framework~\eqref{eq:unified} \\
$ h(\rho,\theta)= \bf 0$ & the implicit dependence of $\rho$ on $\theta$\\
$\lambda_\xi, \lambda$ & the adjoint variable, solutions to the adjoint equation\\
\hline
\end{tabular}
\end{table}

\section{Algorithmic Details Regarding Numerical Implementation}
This section presents more details on the numerical implementation of the NGD methods. In particular, we explain how to obtain the matrix $L$ in~\eqref{eq:unified} for the WNGD~\eqref{eq:nat_grad_W2_L2} in~\Cref{sec:WNGD}. We have proposed in~\Cref{subsec:Z_avail} that the QR factorization could efficiently solve the least-squares problem~\eqref{eq:unified}. In~\Cref{sec:qr_low_rank}, we discuss how to handle rank deficiency in $Y = LZ$ through the QR factorization. 

The main difficulties of computing NGD for large-scale problems include no direct access to the Jacobian $Z$ (see~\Cref{subsec:Z_unavail}) and the computational cost of handling $Z$ even if it is directly available. Here, we present two interesting ideas that may mitigate these challenges, although we have not thoroughly investigated them in the context of NGD methods. We discuss in~\Cref{subsec:Z_unavail_Hut} one strategy based on randomized linear algebra if the Jacobian $Z$ is unavailable. In~\Cref{sec:column_Z}, we briefly comment on an idea to further reduce the computational complexity of the NGD methods by possibly obtaining a low-rank approximation of the Jacobian $Z$.

\subsection{More Discussions on Computing the Wasserstein Natural Gradient}\label{sec:WNGD}
As explained in~\Cref{subsec:W2nat_math}, the Wasserstein tangent vectors at $\rho$ are velocity fields of minimal kinetic energy in $L^2_\rho(\R^d;\R^d)$. After a change of variable, $\Tilde v_i=\sqrt{\rho}~v_i$ and $\Tilde v_i$ satisfies~\eqref{eq:zeta_to_v}. We will discuss next how to solve this minimization problem numerically.

\textbf{Discretization of the divergence operator.}
To compute the Wasserstein natural gradient, the first step is to solve~\eqref{eq:zeta_to_v}, which becomes~\eqref{eq:obtain v's} after discretization. 
\begin{equation}\label{eq:obtain v's}
\min_{y} \|y\|_2^2 \quad \text{s.t.}~B y =\zeta_i,\quad i = 1,\ldots,p.
\end{equation}
If the domain $\Omega$ is a compact subset of $\R^d$ (in terms of numerical discretization), the divergence operator in~\eqref{eq:zeta_to_v} comes with a zero-flux boundary condition. That is, $\tilde{v} = 0$ on $\partial \Omega$. For simplicity, we describe the case $d=2$ where $\Omega$ is a rectangular cuboid. All numerical examples we present earlier in this paper belong to this scenario.

First, we discretize the domain $ [\mathfrak{a},\mathfrak b] \times [\mathfrak c,\mathfrak d]$ with a uniform mesh with spacing $\Delta x$ and $\Delta y$ such that $x_{0} = \mathfrak{a}$, $x_{n_x}  = \mathfrak{b}$, $y_{0} = \mathfrak{c}$, and $y_{n_y}  = \mathfrak{d}$. The left-hand side of the linear constraint in~\eqref{eq:zeta_to_v} becomes a matrix 
$$
B=  - \begin{bmatrix} A_x D & A_y D\end{bmatrix}
$$ 
in~\eqref{eq:obtain v's} 
where $D  = \text{diag}(\sqrt{\Vec{\rho}})$,
$A_x = \frac{1}{2\Delta x} C_{n_x-1} \otimes  I_{n_y-1}  $ and $A_y = \frac{1}{2\Delta y}  I_{n_x-1} \otimes  C_{n_y-1}$. Here, $\Vec{\rho}$ is a vector-format discretization of the function $\rho$ while skipping the boundary points,  $\otimes$ denotes the Kronecker product, $I_n \in \R^{n\times n}$ is the identify matrix and $C_n\in\R^{n\times n}$ is the central difference matrix with the zero-Dirichlet boundary condition.
\begin{equation}\label{eq:central_diff}
 C_n =    \begin{bmatrix} 
    0 & 1 & \\
    -1 & 0 & 1 & \\
      & \ddots & \ddots & \ddots \\
      & &-1 & 0  & 1 \\
     & &  & -1 & 0 
    \end{bmatrix}_{n\times n}.
\end{equation}

One may also use a higher-order discretization for the divergence operator in~\eqref{eq:zeta_to_v}. The discretization of the vector field $\tilde v = (\tilde v_x, \tilde v_y)^\top$ is $y = (y_1^\top, y_2^\top)^\top$ in~\eqref{eq:obtain v's} where $y_1$  and $y_2$ are respectively the vector-format of $\tilde v_x$ and $\tilde v_y$ while skipping the boundary points due to the zero-flux boundary condition. Note that $B$ is full rank if $\rho$ is strictly positive, and $n_x$, $n_y$ are odd. We remark that $B$ and $y$ remain very similar structures if $\Omega\subset \R^d$ with $d>2$.

\textbf{$Z$ available.} 
If $Z=(\zeta_1~\zeta_2~\ldots~\zeta_p)$ is available, we can solve~\eqref{eq:zeta_to_v} directly. After discretization, these equations reduce to constrained minimum-norm problems~\eqref{eq:obtain v's}, where $B$ is the discretization of the differential operator $-\nabla \cdot (\sqrt{\rho}\ \sbullet[.75])$ evaluated at the current $\theta$ (and thus $\rho(\theta)$). The solution to~\eqref{eq:obtain v's} can be recovered via the pseudoinverse of $B$ as
\begin{equation}\label{eq:B_dagger}
    Y=B^\dagger Z,\quad\text{where}\quad  Y=(\Tilde{v}_1~\Tilde{v}_2~\ldots~\Tilde{v}_p)\quad \text{and}\quad  Z=(\zeta_1~\zeta_2~\ldots~\zeta_p).
\end{equation}
In our case, $B$ is underdetermined and we assume it to have full row ranks. 
We could perform the QR decomposition of $B^\top$ in the ``economic'' size: 
\begin{equation}\label{eq:QR B_dagger}
    B^\dagger = Q (R^{\top} )^{-1},\quad \text{where}\quad B^\top = Q R.
\end{equation}
Since $R^\top$ is lower diagonal, $\Tilde{v}_i = Q(R^{\top} )^{-1}\zeta_i$ can be efficiently calculated via forward substitution. If $p$ is not too large, and we have access to $\{\zeta_i\}$ directly, this is an efficient way to obtain $\{\Tilde v_i\}$.

Once we obtain $Y$, we can compute the $W_2$ NGD direction since~\eqref{eq:nat_grad_W2} reduces to
\begin{equation}\label{eq:vtilde}
        \eta^{nat}_{W_2}=\argmin_{\eta \in \R^p} \bigg \|\sqrt{\rho}~\partial^W_\rho f+\sum_{i=1}^p \eta_i \Tilde{v}_i \bigg\|^2_{L^2(\R^d;\R^d)}=-Y^\dagger \left(\sqrt{\rho}~\partial^W_\rho f\right),
\end{equation}
where $\partial^W_\rho f$ is related to $\partial_\rho f$ based on~\eqref{eq:Was-flat-connection}, and $Y^\dagger$ is the pseudoinverse of $Y$ which one can obtain by QR factorization; see details in~\Cref{subsec:Z_avail}.

We can also compute the $W_2$ information matrix based on $B^\dagger$ obtained via the QR factorization~\eqref{eq:QR B_dagger}. That is,
\[
G_{w_2}=Y^\top Y = Z^\top (BB^\top)^\dagger Z = Z^\top (B^\dagger )^\top B^\dagger Z.
\]
Therefore, if $Y$ has full column ranks, the common approach is to invert the information matrix $G_{w_2}$ directly and obtain the NGD direction following~\eqref{eq:nat_std_relation} as
\[
     \eta^{nat}_{W_2} = - G_{w_2}^{-1}\ \partial_\theta f(\rho(\theta)).
\]

\textbf{Discretization of the Wasserstein Gradient $\partial_\rho^W f$.}
Based on~\eqref{eq:nat_grad_W2_L2}, we need to discretize the weighted Wasserstein Gradient, $b \approx  -\sqrt{\rho} \partial_\rho^W f =  -\sqrt{\rho} \nabla \partial_\rho f$, such that the WNGD $\eta_{W_2}^{nat} = Y^\dagger b$ where $Y = B^\dagger Z$. We remark that the discretization of the gradient operator in $\sqrt{\rho} \nabla  \partial_\rho f(\rho(\theta))$ needs to be the numerical adjoint with respect to the matrix $-B$, the discretization of the divergence operator. That is, 
\begin{equation*}
b \approx  - \sqrt{\rho} \nabla \left( \partial_\rho f(\rho(\theta)) \right) =   (- B) ^\top   \partial_\rho f.
\end{equation*}
This requirement is to ensure that
\begin{align*}
   \partial_{\theta_j}f(\rho(\theta)) \approx \partial_\rho f ^\top  \zeta_j = \partial_\rho f^\top   B y_j  =  \left( B^\top \partial_\rho f \right)^\top y_j  =  - b ^\top y_j \approx \langle \sqrt{\rho} \nabla \partial_\rho f,    \sqrt{\rho} v_j\rangle_{L^2(\R^d;\R^d)},
\end{align*}
which is the discrete version of 
\begin{align*}
    \lim \limits_{t\to 0}\frac{f(\rho+t \zeta)-f(\rho)}{t}=\int_{\R^d} \partial_\rho f(\rho)(x) \, \zeta(x) dx  = \int_{\mathbb{R}^d}  \sqrt{\rho} \nabla \partial_\rho f(\rho)(x) \cdot \Tilde{v}(x) dx, \quad \forall \zeta \in L^2(\R^d).
\end{align*}
The equation above is the main identity used in the proof for~\Cref{prop1}.

For example, if we use the central difference scheme for the divergence operator $-\nabla \cdot (\sqrt{\rho}\ \sbullet[.75])$, we also need to  use  central difference  for the gradient operator $\nabla$. Similarly, if one uses forward difference  for $-B$, the backward difference  should be employed for the gradient operator $\nabla$.

\subsection{Dealing with rank deficiency}\label{sec:qr_low_rank}
Note that in~\eqref{eq:unified} , we need to solve a least-squares problem given the matrix $Y=LZ$ to find the NGD direction based upon a wide range of Riemannian metric spaces. For simplicity, we will consider the problem in its general form: finding the least-squares solution $\eta$ to $Y \eta  = b$ where $b = - (L^\top)^\dagger \partial_\rho f$ based on~\eqref{eq:unified}.

The standard QR approach only applies if $Y$ has full column rank, i.e., $\text{rank}(Y) = p$ while $Y\in\R^{k\times p}$. Otherwise, if $\text{rank}(Y) = r <p$, we are facing a rank-deficient problem, and an alternative has to be applied. Even if $Y$ is full rank, sometimes we may have a nearly rank-deficient problem when the singular values of $Y$, $\{\sigma_i\}$, $i=1,\ldots,p$, decay too fast such that $\sigma_{r+1},\ldots, \sigma_p \ll \sigma_r$. A conventional way to deal with such situations is via QR factorization with column pivoting.

In order to find and then eliminate unimportant directions of $Y$, essentially, we need a \textit{rank-revealing} matrix decomposition of $Y$. While SVD (singular value decomposition) might be the most common choice, it is relatively expensive, which motivated various works on rank-revealing QR factorization as they take fewer flops (floating-point operations) than SVD. The column pivoted QR (CPQR) decomposition is one of the most popular rank-revealing matrix decompositions~\cite{heavner2019building}. We remark that CPQR can be easily implemented in Matlab and Python through the standard \texttt{qr} command, which is based upon LAPACK in both softwares~\cite{lapack99}.

Applying CPQR to $Y$ yields
\[
    YP = QR,
\]
where $P$ is the permutation matrix. Thus, the linear equation $Y \eta = b$ becomes 
\[
    Y PP^\top \eta = QR P^\top \eta = QR \,\eta_p = b,\quad \text{where } \eta_p = P^\top \eta.
\]
Now, we denote by $\widetilde Q$ and $\widetilde R$ the truncated versions of $Q$ and $R$ respectively by keeping the first $r$ columns of $Q$ and the first $r$ rows of $R$. We may solve the linear system below instead 
\[
\widetilde R \eta_p =   \widetilde Q^\top b.
\]

The least-squares solution is no longer unique since we have truncated $R$ due to the (nearly) rank deficiency of $Y$. By convention, one may pick the one with the minimum norm among all the least-squares solutions. Since $\|\eta\|_2 = \|\eta_p\|_2$ as $P$ is a permutation matrix, this is equivalent to finding a minimum-norm solution to the above linear system. This can be done by an additional QR factorization. Let 
$$
\widetilde R^\top = Q_1 R_1
$$ 
where $Q_1\in \R^{p\times r}$ has orthonormal columns and $R_1\in \R^{r\times r}$ is invertible. As a result, 
$$
\eta_p = Q_1 (R_1^\top)^{-1}\widetilde Q^\top b.
$$ 
Finally, we may obtain the solution
\[
\eta = P\eta_p = PQ_1 (R_1^\top)^{-1}\widetilde Q^\top b.
\]
Again, $(R_1^\top)^{-1}$ should be understood as forward substitution.

We may apply the same idea if $B$ in~\eqref{eq:obtain v's} is (nearly) rank deficient while we will keep its dominant $r$ ranks. Note that $B$ is short wide. Applying CPQR to $B^\top$ yields
\[
    B^\top P = QR,
\]
where $P$ is the permutation matrix, $Q$ has orthonormal columns, and $R$ is a $p\times p$ square matrix. Thus, the constraint in~\eqref{eq:obtain v's} becomes
\[
    PP^\top By =  P R^\top Q^\top  y = \zeta_i.
\]
Again, we denote by $\widetilde Q$ and $\widetilde R$ the truncated version of $Q$ and $R$ by keeping the first $r$ columns of $Q$ and the first $r$ rows of $R$ where $r \leq p$. We may solve the linear system below instead 
\[
\widetilde R^\top y_q =  P^\top \zeta_i,\quad \text{where } y_q = \widetilde Q^\top  y.
\]
Since $\widetilde R^\top $ is tall skinny, we may select the least-squares solution to the above system. We perform a QR decomposition in economic size for $\widetilde R^\top$ such that $\widetilde R^\top = Q_2 R_2$. Therefore, 
\[
    y_q = R_2^{-1} Q_2^\top P^\top \zeta_i,
\]
and eventually leads to
\[
    \Tilde{v}_i = y = \widetilde Q y_q =\widetilde Q R_2^{-1} Q_2^\top P^\top \zeta_i.
\]
Note that if $\widetilde R = R$ and $\widetilde Q = Q$, i.e., $r=p$, the solution above coincides with the one obtained from~\eqref{eq:B_dagger}-\eqref{eq:QR B_dagger} since $R_2^{-1} Q_2^\top = (R^\top )^{-1}$.

To sum up, for a tall-skinny matrix $Y$, we compute the following by two QR factorizations while eliminating the unimportant directions during the process:
\[
    Y P = \widetilde Q R_1^\top Q_1^\top,
\]
where $R_1$ is a invertible square matrix while $\widetilde Q$ and $Q_1$ have orthonormal columns. Therefore,
\[
    Y^\dagger = P Q_1 (R_1^\top)^{-1}  \widetilde Q^\top.
\]
Finally, $\eta =Y^\dagger b =  P Q_1 (R_1^\top)^{-1}  \widetilde Q^\top b$. For a short-wide matrix $B$, we compute 
\[
    B^\top P = \widetilde Q R_2^\top Q_2^\top,
\]
where $R_2$ is invertible while  $\widetilde Q$ and $Q_2$ have orthonormal columns. Consequently,
\[
    B^\dagger =  \widetilde Q R_2^{-1}  Q_2^\top  P^\top.
\]
Finally, $\Tilde{v}_i = B^\dagger \zeta_i = \widetilde Q R_2^{-1}  Q_2^\top   P^\top \zeta_i$, for $i=1,\ldots,p$.

\subsection{$Z$ not available: the Hutchinson method}\label{subsec:Z_unavail_Hut}
In this subsection, we present some ideas of approximating $Z$ using Hutchinson's estimator~\cite{hutchinson1990,meyer2020hutch,yao2021adahessian}, a powerful technique from randomized linear algebra. Let $\xi\in \R^k$ be a vector with i.i.d.~random coordinates of mean $0$ and variance $1$. Such random vectors serve as a random basis. That is, 
\[
    Z=\E \left[  \xi \xi^\top Z \right].
\]
Thus, if we have $m$ such random vectors, $\xi_1,\xi_2,\ldots,\xi_m$, then we can estimate
\[
    H_m(Z)=\frac{1}{m}\sum_{k=1}^m \xi_k \xi_k^\top Z.
\]
Furthermore, by introducing the adjoint variables $\lambda_1,\lambda_2,\ldots,\lambda_m$ such that
\begin{equation}\label{eq:lambda_h_adjoint}
    \lambda_k^\top \partial_\rho h = \xi_k^\top,\quad 1\leq k \leq m,
\end{equation}
and using~\eqref{eq:adj_grad}, we obtain
\[
    H_m(Z)= -\frac{1}{m} \sum_{k=1}^m \xi_k \lambda_k^\top \partial_\theta h.
\]
Hence, by replacing $Z$ in \eqref{eq:unified} with its approximation $H_m(Z)$, we obtain an approximated NGD direction as
\begin{equation}\label{eq:unified_discrete_Hutch}
    \eta^{nat}_L =\argmin_{\eta \in \R^p} \big \| (L^\top)^\dagger \partial_\rho f+L\ H_m(Z) \ \eta \big \|_2^2.
\end{equation}
Once we obtain $H_m(Z)$, the above least-squares problem can be solved by QR factorization, similar to the framework presented in~\Cref{subsec:Z_avail} or~\Cref{sec:qr_low_rank}. However, we remark here that the convergence behavior of $H_m(Z) \xrightarrow{m\rightarrow \infty} Z$ depends on the spectral properties of $Z$.

\subsection{Exploring the column space of $Z$ implicitly}\label{sec:column_Z}
As discussed in~\Cref{subsec:Z_unavail_Hut}, one way to reduce the complexity of implementing the NGD method is to find a low-rank approximation to the Jacobian $Z = \partial_\theta \rho$. For any $\zeta$, we have that $\zeta=\E \left[\langle \zeta, \xi \rangle \xi \right]$
given any random vector $\xi$ whose covariance is the identity. Hence, by the law of large numbers, for $m$ large enough, we have that
\begin{equation}\label{eq:zeta_approx}
    \mathbb{P} \left( \left\|\zeta - \hat{\zeta} \right\| > \epsilon \right) < \delta,\quad \text{where } \hat{\zeta} = \frac{1}{m} \sum_{k=1}^m \langle \zeta,\xi_k \rangle \xi_k,
\end{equation}
where  $\{\xi_1,\xi_2,\cdots,\xi_m\}$ are i.i.d.~random vectors. Therefore,
\[
    \left\| L \left(\zeta_j -  \hat{\zeta}_j\right)\right\| < \|L\| \epsilon,\quad 1\leq j \leq p,
\]
with high probability when $m$ is large enough (depending on the spectral property of $Z$). Here, $L$ is the important linear operator in the unified framework~\eqref{eq:unified}. In~\Cref{subsec:Z_unavail_Hut}, we approximate 
\[
Y = LZ \approx L H_m(Z),
\]
which is to compute the approximation matrix $H_m(Z)$ directly. Next, we present another way to obtain an approximated $Y$ whether or not $Z$ is explicitly available.

If we can find such $\{\xi_k\}$ satisfying~\eqref{eq:zeta_approx}, our final approximation to each $y_j$ in $Y = L Z =  (y_1 \ldots y_j \ldots,y_p)$ could be written as
\begin{equation}\label{eq:vj_Hutch}
  y_j = L \zeta_j \approx L \hat{\zeta}_j=  \frac{1}{m} \sum_{k=1}^m \langle \zeta_j,\xi_k \rangle\, L \xi_k,\quad 1\leq j \leq p. 
\end{equation}
Note that the inner product $\langle \zeta_j, \xi_k \rangle$ can be computed via the adjoint-state method if there is no direct access to $\{\zeta_j\}$; see~Section~\ref{subsec:adj_implicit} for details. Therefore, to obtain an approximated $Y$, we only need to evaluate $L h_k$ and the inner products $\langle \zeta_j, \xi_k \rangle$ for each $k$ and $j$, without directly accessing the Jacobian $Z =( \zeta_1 \ldots  \zeta_p)$.  A similar idea called randomized SVD could also apply here~\cite{halko2011finding}.

\clearpage

\end{document}